\documentclass[12pt]{amsart}

\usepackage{fancyhdr}
\usepackage[utf8]{inputenc}
\usepackage[english]{babel}

\usepackage{latexsym,color,amsmath,amsthm,amssymb,amscd,amsfonts}
\usepackage{tikz}
\usetikzlibrary{arrows}
\usepackage{scalefnt}
\usepackage[font=small,labelfont=bf]{caption}
\usepackage{float}
\usepackage{graphicx}
\usepackage{relsize}
\usepackage{amsopn}
\usepackage{cancel}
\usepackage{mathrsfs}  
\usepackage[hidelinks]{hyperref}
\usepackage{bm}
\usepackage{bbm}
\usepackage[shortlabels]{enumitem}
\usepackage{multicol}
\usepackage{float} 
\usepackage{todonotes}
\usepackage{cite}
\usepackage{soul}

\usepackage{array}

\usepackage[margin=2.5cm]{geometry}

\newtheoremstyle{nonum}{}{}{\itshape}{}{\bfseries}{.}{ }{\thmnote{#3}}

\newtheorem{theorem}{Theorem}[section]
\newtheorem*{theorem*}{Theorem}

\newtheorem{corollary}[theorem]{Corollary}
\newtheorem{lemma}[theorem]{Lemma}

\newtheorem{proposition}[theorem]{Proposition}

\newtheorem{example}[theorem]{Example}
\newtheorem{definition}[theorem]{Definition}
\newtheorem*{definition*}{Definition}
\newtheorem{remark}[theorem]{Remark}
\newtheorem*{remark*}{Remark}
\newtheorem*{question*}{Question}

\theoremstyle{nonum}

\newcommand{\R}{\mathbb R}

\def\supp{{\rm supp}}

\def\R{\mathbb R}

\title{On $p$-Brunn--Minkowski and Brascamp--Lieb inequalities}

\author{Alexander Kolesnikov}
\address{HSE University, Russian Federation}
\email{sascha77@mail.ru}

\author{Galyna Livshyts}
\address{Georgia Institute of Technology, USA, and Technion - Israel Institute of Technology, Israel}
\email{glivshyts6@math.gatech.edu}

\author{Liran Rotem}
\address{Technion - Israel Institute of Technology, Israel}
\email{lrotem@technion.ac.il}

\begin{document}

\begin{abstract}

We show that a strong version of the Brascamp--Lieb inequality for symmetric log-concave 
measure with $\alpha$-homogeneous potential $V$ is equivalent to a $p$-Brunn--Minkowski inequality for level sets of $V$ with some $p(\alpha,n)<0$. We establish links between several inequalities of this type on the sphere and in Euclidean space.
Exploiting these observations, we prove new sufficient conditions for the symmetric $p$-Brunn--Minkowski inequality with $p<1$. In particular, we prove the local log-Brunn--Minkowski inequality for $l^q_n$-balls for all $q\geq 1$ in all dimensions, which was previously known only for $q\geq 2$.


\end{abstract}

\maketitle

\section{Introduction}

The  Brascamp--Lieb inequality, discovered by Brascamp and Lieb in \cite{BrLi} states that every log-concave probability measure $\mu$ with density $\frac{e^{-V} }{\int_{\mathbb{R}^n} e^{-V}dx}$
on $\mathbb{R}^n$ 
satisfies the following Poincar{\'e}-type inequality:
\begin{equation}
    \label{BL1}
{\rm Var}_{\mu} f:= \int_{_{\mathbb{R}^n}} f^2 d\mu - \Bigl( \int_{_{\mathbb{R}^n}} f d\mu \Bigr)^2 \le \int_{_{\mathbb{R}^n}} \langle (D^2 V)^{-1} \nabla f, \nabla f \rangle d\mu.
\end{equation}
Normally, we assume that $V$ is a proper convex function, $\{ V < \infty\}$ has a nonempty interior, and for $\mu$-almost all $x$, the function $V$ is twice continuously differentiable at $x$, and $D^2 V (x)$ is nondegenerate.
In particular, under these assumptions, both sides of the inequality (\ref{BL1}) are well-defined for any continuously differentiable $f$, and (\ref{BL1}) holds.
Note that the  integrability of entries $(D^2 V)^{-1}_{ii}$ is, in general, not assumed; thus, the right-hand side of (\ref{BL1}) might be  infinite.  We would prefer to work with smooth, non-degenerate potentials $V$, but we can't restrict ourselves completely to this case because, in this paper, we deal with homogeneous potentials, which are, in general, non-smooth and  degenerate (at least at the origin). 

Inequality (\ref{BL1}) can be  obtained by differentiating the Pr\'ekopa--Leindler inequality (see Brascamp--Lieb \cite{BrLi} and Bobkov--Ledoux \cite{BL1-BrLib}), which is known to be the functional  version of the Brunn--Minkowski inequality 
\begin{equation}
    \label{BM1-mult}
|\lambda  A+ (1- \lambda)  B|
\ge |A|^{\lambda}|B|^{1-\lambda},
\end{equation}
where $A, B$ are Borel subsets of $ \mathbb{R}^n $, $\lambda \in [0,1]$ and $|A|$ is the Lebesgue volume of $A$. The equivalent additive version of this inequality is
\begin{equation}
    \label{BM1}
|\lambda  A+ (1- \lambda)  B|^{\frac{1}{n}}
\ge \lambda |A|^{\frac{1}{n}}+(1-\lambda)|B|^{\frac{1}{n}}.
\end{equation}

The famous logarithmic Brunn--Minkowski conjecture is one of the most interesting open problems in convex geometry.
Instead of Minkowski summation,  we consider another natural operation on convex sets, the so-called $p$-summation, where $p$ is a real parameter. This summation was introduced by W.J. Firey \cite{Firey} in order to extend the classical Brunn--Minkowski theory. For values $p \ge 1$, one can simply define the set
$\lambda \cdot A +_p (1-\lambda) \cdot B$ by choosing its support functional in the following way:
$$
h^p_{\lambda \cdot A +_p (1-\lambda) \cdot B} = \lambda h^p_A + (1-\lambda ) h^p_B.
$$
Note that for $p=1$ we get the usual Minkowski addition.
If $p<1$, the function $\lambda h^p_A + (1-\lambda ) h^p_B$ can be non-convex, and the definition of $p$-addition  $\lambda \cdot A +_p (1-\lambda) \cdot B$  requires some accuracy: 
$$
\lambda \cdot A +_p (1-\lambda) \cdot B=\cap_{u\in\mathbb{S}^{n-1}}\Big\{x\in\R^n:\, \langle x,u\rangle \leq \left(\lambda h^p_A(u) + (1-\lambda ) h^p_B(u)\right)^{\frac{1}{p}}\Big\}.
$$
See more details,e.g. in B\"or\"oczky--Figalli--Ramos \cite{BFR}.

We say that convex sets $A,B$ satisfy the $p$-Brunn--Minkowski (or simply $p$-BM inequality) if
\begin{equation}
\label{BMp}
|\lambda \cdot A +_p (1-\lambda) \cdot B|^{\frac{p}{n}}
\ge \lambda |A|^{\frac{p}{n}} + (1-\lambda) |B|^{\frac{p}{n}}.
\end{equation}
It was shown by Firey \cite{Firey}  that (\ref{BMp}) holds for all convex bodies containing the origin and $p\ge 1$.
As in the classical case corresponding to $p=1$, (\ref{BMp}) implies the uniqueness of a solution to a variant of the Minkowski problem, the so-called $p$-Minkowski problem (see explanations in \cite{BFR}, Remark 9.47).

Unlike $p>1$, the case $p<1$ contains many open questions. First, it is well-known that (\ref{BMp}) fails even in the class of convex sets. A natural assumption for (\ref{BMp}) is that $A,B$ are {\bf symmetric convex} sets. 
For the non-symmetric case, the $p$-Brunn-Minkowski inequality with $0<p<1$ is, in general, not true, because it is known that the corresponding $p$-Minkowski problem can have many solutions (see Chen--Li--Zhu \cite{CLZ}). For $p<0$ the inequality (\ref{BMp}) fails even for symmetric sets (see \cite{BFR}, Subsection 9.4,  Chen--Li-Zhu \cite{CLZ},  Li--Liu--Lu \cite{LLL}).

It was conjectured in the seminal paper \cite{BLYZ} of B{\"o}r{\"o}czky--Lutwak--Zhang--Yang that (\ref{BMp}) holds for all symmetric convex sets and $p \in [0,1)$. Note that the validity of (\ref{BMp}) for some $p=p_0$ implies (\ref{BMp}) for all $p >p_0$; thus, the strongest version of the conjectured inequality (\ref{BMp}) is the $0$-BM inequality, which is understood in the limiting sense and is called the logarithmic Brunn--Minkowski (log Brunn--Minkowski) inequality.

In this paper, we work only with a local version of (\ref{BMp}). In the case $p=0$, this local version was derived by Colesanti, the second-named author, and Marsiglietti \cite{CLM} (Theorem 6.5), and for other $p$, it was derived by Emanuel Milman and the first-named author in \cite{KM-LpBMproblem}. It was shown later in (Chen--Huang--Li--Liu  \cite{CHLL}, Putterman \cite{Putt}) that the local version is equivalent to the global one.

The reduction of the Brunn–Minkowski inequality to its local version was already used by D. Hilbert in his proof of the classical Brunn–Minkowski inequality, but this fact seems to be not well-known, and the local Brunn–Minkowski inequality was rediscovered by  Colesanti  in \cite{Colesanti}. The general local $p$-BM inequality established in \cite{KM-LpBMproblem} has the form 
\begin{equation}
    \label{240125}
(n-p) {\rm Var}_{\nu^*} f \le \int_{\mathbb{S}^{n-1}}
\big\langle \Bigl( I + \frac{\nabla^2_{\mathbb{S}^{n-1}} h}{h}\Bigr)^{-1} \nabla_{\mathbb{S}^{n-1}} f, \nabla_{\mathbb{S}^{n-1}} f \big\rangle d\nu^*.
\end{equation}
Here $f$ is a sufficiently smooth even function on $\mathbb{S}^{n-1}$, $\nabla_{\mathbb{S}^{n-1}}$
and $\nabla^2_{\mathbb{S}^{n-1}}$ are the spherical gradient and the spherical Hessian. Function $h = h_K$ is a {\bf support} function of a symmetric strictly convex body $K$:
$$
h(\theta) = h_{K}(\theta) = \sup_{x \in K} \langle \theta,x \rangle, 
$$
$$D^2 h = h \cdot I + \nabla^2_{\mathbb{S}^{n-1}} h = {\rm II}^{-1}_{\partial K}(n^{-1}_K(\theta)) $$ is the inverse second fundamental form of $\partial K$ at the point $$n^{-1}_K(\theta) = h \cdot \theta + \nabla_{\mathbb{S}^{n-1}} h,$$ where $\theta \in \mathbb{S}^{n-1},$ $n_K \colon \partial K \to \mathbb{S}^{n-1}$ is the Gauss map on $\partial K$. Finally,
$$
\nu^* = \frac{h \det D^2h d\theta}{\int_{\mathbb{S}^{n-1}} h \det D^2h d\theta}
$$
is the so-called cone measure of $K$ and $d\theta$ is the $(n-1)$-dimensional Hausdorff measure on $\mathbb{S}^{n-1}$.

\begin{remark}
    By a slight abuse of notation, we will use the same symbol $D^2$ for the Euclidean Hessian (matrix of second derivatives) $D^2 f$ of a function $f$ on $\mathbb{R}^n$ and the operator
    $D^2 h = h \cdot I + \nabla^2_{\mathbb{S}^{n-1}} h$  on functions on $\mathbb{S}^{n-1}$. Note, however, that if $f(x) = r \cdot h(\theta)$ is the $1$-homogeneous extension of $h$, then ``Euclidean'' $D^2 f$ coincides with ``spherical'' $D^2 h$ on all tangent spaces to $\mathbb{S}^{n-1} = \{x: |x|=1\}$.
\end{remark}

\begin{definition}\label{local-Lp-BM-body} We say that a symmetric convex body $K$ satisfies the local $p$-Brunn-Minkowski inequality if (\ref{240125}) holds, with $h$ being the support function of $K$ and $\nu^*$ defined as above. If $p=0$, we say that $K$ satisfies the local log-Brunn-Minkowski inequality.
\end{definition}

Let us make a short overview of the results about $p$-Brunn--Minkowski inequality. The log-Brunn--Minkowski conjecture $(p=0)$ for symmetric sets was advertised and verified for the plane in  B\"or\"oczky--Lutwak--Yang--Zhang \cite{BLYZ} (their result was slightly generalized in Ma \cite{Ma}, and an alternative proof of the planar case was given in Putterman \cite{Putt}). The necessary and sufficient conditions for the existence of  a solution to the log-BM problem were also outlined in \cite{BLYZ}. The third-named author has proved  the log-BM conjecture for complex convex bodies \cite{Rotem}. {Saroglou \cite{Saroglou1}, \cite{Saroglou2} extended the consideration of the problem to all log-concave measures and studied relations to the  celebrated B-conjecture from the 90s, popularized by Latala \cite{lat}. In particular, he showed that an affirmative solution to the
log-Brunn--Minkowski conjecture implies an affirmative solution to the B-conjecture. In addition, an affirmative solution to the strong B-conjecture for the uniform measure on the cube (in any dimension) implies an affirmative solution to the log-Brunn--Minkowski problem.}
In fact, before it was conjectured, the log-Brunn-Minkowski inequality for unconditional convex bodies had already been verified by Cordero-Erausquin, Fradelizi, Maurey \cite{CFM}, and was later re-proved by Saroglou. Colesanti--Livshyts--Marsiglietti \cite{CLM} derived the local form of the log-Brunn-Minkowski conjecture from the global form, and verified that it is correct for balls, while Kolesnikov--Milman \cite{KM} proved the local $p$-Brunn-Minkowksi to be true for all symmetric convex sets with 
$p \in  (p(n),1]$, where $p(n)<1$ and $\lim_{n \to \infty} p(n) = 1$. They also proved it for some special types of sets, including $l^p_n$-balls within some range of $p$ depending on $n$, extending a result of Colesanti--Livshyts--Marsiglietti \cite{CLM}, which was only done for the case of the euclidean ball. The equivalence of the global and local forms of the  $p$-Brunn--Minkowski inequality was proved by  Chen--Huang--Li--Liu \cite{CHLL} and Putterman \cite{Putt}. Van Handel proved the local form of the log-Brunn--Minkowski inequality for zonoids \cite{Handel}, which in particular extended to all dimensions the result of Kolesnikov and Milman \cite{KM} regarding the $l^p_n$-balls. Some partial results for the $p$-Brunn--Minkowski inequality for measures have been obtained in Hosle--Kolesnikov--Livshyts \cite{HKL}. E. Milman \cite{Milman-cageometry}  suggested a new geometrical viewpoint on the problem based on the notion of centro-affine connection. He proved, in particular, that the local $p$-BM inequality with ($ p =  p(n,\frac{\lambda}{\Lambda})$) holds, provided $K$ satisfies an inequality of the type
$$
\lambda \le  II_{\partial K} \le \Lambda
$$
for some appropriate value of the ratio 
$\frac{\lambda}{\Lambda}$. Here $II_{\partial K}$ is the second fundamental form of $\partial K$. See further developments in Ivaki--Milman \cite{MI1}. 
Other related results can be found in  \cite{KolLyv}, \cite{NT}, \cite{BorKal}, \cite{XiL}, \cite{Stancu}.
For a  detailed overview, see in \cite{Bor-survey}, \cite{BFR}.

Despite the well-known fact that inequalities (\ref{BL1}) and (\ref{240125}) (for $p=1$) are infinitesimal versions of the Brunn--Minkowski inequality, the direct description of the relation between inequalities of  both types seems to be missing in the literature.
The authors were only aware of the implication
$(\ref{BL1})  \Longrightarrow  (\ref{240125})$ for $p=1$ communicated to us by Dario Cordero-Erausquin. But what happens for other values of $p$? This seems to be a natural question, which we will turn to address in the sequel.

To answer this question, let us first rewrite inequality (\ref{240125}) in terms of the Minkowski functional $\phi$ of the body $K$ (instead of the support function $h$):
$$
\phi(\varphi) = \phi_K(\varphi) = \inf_{t > 0,  \varphi \in t K} t.
$$
To do this, we parametrize $\mathbb{S}^{n-1}$ by variables $\varphi$ and $\theta$ and  consider the following couple of mappings $S  \colon \mathbb{S}^{n-1} \to \mathbb{S}^{n-1}, \ T \colon \mathbb{S}^{n-1} \to \mathbb{S}^{n-1}$:
\begin{equation}
\label{sandt}
     S(\varphi) = \frac{\phi\cdot \varphi + \nabla_{\mathbb{S}^{n-1}} \phi}{\sqrt{\phi^2 + |\nabla_{\mathbb{S}^{n-1}} \phi|^2}}, \ \ \ 
    T(\theta) = \frac{h\cdot \theta + \nabla_{\mathbb{S}^{n-1}} h}{\sqrt{h^2 + |\nabla_{\mathbb{S}^{n-1}} h|^2}}.
\end{equation}
It can be verified by direct computations (see more details in Section \ref{dualBL}) that $T$ and $S$ are inverse to each other, and $T$ pushes forward $\nu^*$ onto
 $$
       \nu = \frac{\frac{d \varphi}{\phi^n}}{\int_{\mathbb{S}^{n-1}}\frac{d \varphi}{\phi^n}}.
  $$
Moreover, (\ref{240125}) takes the following form in the $\varphi$-coordinate system:
\begin{equation}
    \label{240126}
(n-p) {\rm Var}_{\nu} g \le \int_{\mathbb{S}^{n-1}}
\big\langle \Bigl( I + \frac{\nabla^2_{\mathbb{S}^{n-1}} \phi}{\phi}\Bigr)^{-1} \nabla_{\mathbb{S}^{n-1}} g, \nabla_{\mathbb{S}^{n-1}} g \big\rangle d\nu.
\end{equation}

\begin{remark}
    The dual metrics $I + \frac{\nabla^2_{\mathbb{S}^{n-1}} h
    }{h}$, $I + \frac{\nabla^2_{\mathbb{S}^{n-1}}\phi
    }{\phi}$ on $\mathbb{S}^{n-1}$
have been studied by E.~Milman \cite{Milman-cageometry} in the context of  centroaffine geometry (see, in particular, Section 4 and the references therein).
\end{remark}

The central result of our work relates inequalities of the type (\ref{240126}) to the family of {\bf strong} Brascamp--Lieb inequalities (see Sections 2-3).
We summarize the results in the following theorem.

\begin{theorem}
\label{mainth1}
Let $\Phi$ be a strictly convex, even, $\alpha$-homogeneous potential:
$$\Phi(x) = \frac{1}{\alpha}|x|^{\alpha} \phi^{\alpha}(\varphi),
$$
where  $\alpha >1$, $\varphi = \frac{x}{|x|} \in \mathbb{S}^{n-1}$, and $\phi$ is a twice continuously differentiable function on $\mathbb{S}^{n-1}$. 
Consider probability measures $$\mu = \frac{e^{-\Phi} dx}{\int_{\mathbb{R}^n} e^{-\Phi} dx}, \ \ \nu = \frac{\frac{d\varphi}{\phi^n}}{\int_{\mathbb{S}^{n-1}}\frac{d\varphi}{\phi^n}}$$
on $\mathbb{R}^n$ and $\mathbb{S}^{n-1}$ respectively.

\begin{itemize}
    \item 
Assume that     
\begin{equation}
    \label{calpha}
{\rm Var}_{\mu} f \le C_{\alpha} \int_{\mathbb{R}^n} \langle \bigl( D^2 \Phi \bigr)^{-1} \nabla f, \nabla f \rangle d\mu
\end{equation}
    for some value $ 1 - \frac{1}{\alpha} \le C_{\alpha} \le  1$ and all even $f$.
Then 
 $\nu$  satisfies 
 the following inequality:
 $${\rm Var}_{\nu} g \le \frac{C^2_{\alpha}}{(n-\alpha)C_{\alpha} + \alpha-1} \int_{\mathbb{S}^{n-1}} \langle \phi\bigl( D^2 \phi \bigr)^{-1} \nabla_{\mathbb{S}^{n-1}}  g, \nabla_{\mathbb{S}^{n-1}}  g \rangle d\nu
 $$
 for arbitrary smooth even $g$.
 \item Assume that 
 \begin{equation}
     \label{cnu}
{\rm Var}_{\nu} g \le C_{\nu} \int_{\mathbb{S}^{n-1}} \langle \phi\bigl( D^2 \phi \bigr)^{-1} \nabla_{\mathbb{S}^{n-1}}  g, \nabla_{\mathbb{S}^{n-1}}  g \rangle d\nu
 \end{equation}
for every even function $g$.
    Then $\mu$ satisfies 
    $$
{\rm Var}_{\mu} f \le \max\Bigl( 1 - \frac{1}{\alpha}, n C_{\nu}\Bigr) \int_{\mathbb{R}^n} \langle \bigl( D^2 \Phi \bigr)^{-1} \nabla f, \nabla f \rangle d\mu
    $$
for every even $f$.
\end{itemize}
In particular, (\ref{calpha}) holds with $C_{\alpha} = 1 - \frac{1}{\alpha}$ if and only if (\ref{cnu}) holds with $C_{\nu} = \frac{1}{n} \Bigl( 1 - \frac{1}{\alpha} \Bigr)$.
\end{theorem}

\begin{corollary}
\label{0502}
    Assume that the  probability measure $\mu = \frac{e^{-\Phi} dx}{\int_{\mathbb{R}^n} e^{-\Phi} dx}$, where $\Phi(x) = \frac{1}{\alpha}|x|^{\alpha} \phi^{\alpha}(\varphi),
$
 $\alpha >1$, satisfies inequality (\ref{calpha})  for some value $ 1 - \frac{1}{\alpha} \le C_{\alpha} \le  1$ and all even $f$. Then the set $K = \{\phi \le 1\}$ satisfies the local $p$-Brunn--Minkowski inequality with 
 $$
p= n - \frac{(n-\alpha)C_{\alpha} + \alpha -1}{C^2_{\alpha}}.
 $$
\end{corollary}

As an immediate consequence of Theorem C from the work \cite{CKLR} of Colesanti--Kolesnikov--Livshyts--Rotem   and Corollary \ref{0502}, we get the following result.

\begin{corollary}
\label{lpballres}
The local log-Brunn-Minkowski inequality in $\R^n$ holds when the convex body is 
    $l^q_n$-ball, for all $q \ge 1$ and all $n\geq 1$.
\end{corollary}

We remark that, formally, the potentials in corollary (\ref{lpballres}) are not smooth, as required, but following the proof, the reader can verify,  that the implication works also in this case.

Some partial (non-sharp) results for $l^q_n$-balls have been established in \cite{KM-LpBMproblem}. 
The local log-BM inequality was known for zonoids (see \cite{Handel}), and in particular, for $l^q_n$-balls with $q\ge 2$.  Our results complete this picture. Our results complete this picture.

\begin{remark}
\label{lpbest}
    The result of \cite{CKLR} is actually more precise and relies on the explicit computation of eigenvalues for certain operators on $\mathbb{S}^{n-1}$ (see Section 8 in \cite{CKLR}) which coincide (modulo a coordinate change) with the Hilbert operators for $l^q_n$-balls introduced in \cite{KM-LpBMproblem}. Using this result, we 
can prove that $l^q_n$-balls  satisfy $-\frac{n}{q-1}$-BM inequality for $q \ge 2$ and $(-n + 2(2-q))$-BM inequality for $1 \le q \le 2$. This follows from the proof of Theorem C \cite{CKLR}, where the explicit eigenfunctions were found.
    \end{remark}

Inequalities (\ref{calpha}), (\ref{cnu}) can be viewed as spectral gap inequalities for appropriate differential operators.  More generally, for every pair of probability measures with positive densities 
$$\mu = \frac{e^{-V} dx}{\int_{\mathbb{R}^n} e^{-V} dx}, \ \ \sigma =  \frac{e^{-W} dy}{\int_{\mathbb{R}^n} e^{-W} dy}
$$ on $\mathbb{R}^n$ one can consider the metric-measure space
$$
(\mathbb{R}^n, \mu, D^2 \Phi),
$$
equipped with the probability measure $\mu$ and the Riemannian metric $g= D^2 \Phi$,
where $\nabla \Phi$ is the optimal transportation mapping (Brenier map) of $\mu$ onto $\sigma$. Metrics of this type are called Hessian metrics. 
Hessian metrics appeared in the pioneering works of E.~Calabi (see \cite{Calabi} and references therein) as a natural framework for studying the regularity properties of the solutions to the Monge--Amp\`ere equation. In differential geometry, they are real representatives of important complex metrics on  K{\"a}hler manifolds.
Apart from purely geometrical applications (Minkowski-type problems, K\"ahler--Einstein equation, toric geometry, lattice polytopes, etc.) applications of Hessian metrics include  information geometry and statistics (\cite{Amari}, \cite{Shima}, \cite{PW}), optimal transportation (\cite{Kolesnikov2014}, \cite{KK-curvature}), various geometric and probabilistic inequalities (\cite{Cordero-Klartag}, \cite{Klartag-lcm}, \cite{KK-eigenvalues}, \cite{CFP}, \cite{Caglar-Kolesnikov-Werner}, \cite{CKLR}), Stein kernels  (\cite{Fathi}, \cite{FM}, \cite{Diez}),  computational methods (\cite{DKP}),
thermodynamics and chemical reactions (\cite{KLKS}, \cite{BLMN}).

The mapping $x \to \nabla \Phi(x)$
is the measure preserving metric isomorphism between $
(\mathbb{R}^n, \mu, D^2 \Phi),
$
and $
(\mathbb{R}^n, \sigma, D^2 \Psi),
$
where $\Psi(y) = \Phi^*(y) = \sup_{x \in \mathbb{R}^n} (\langle x, y \rangle - \Phi(x))$ is the Legendre transform of $\Phi$, and $\nabla \Psi$ is the optimal transportation mapping of $\sigma$ to $\mu$. 
The corresponding Dirichlet form
$$
\Gamma(f) = \int_{\mathbb{R}^n}
\langle (D^2 \Phi)^{-1} \nabla f, \nabla f \rangle d\mu
$$
admits the generator
\begin{equation}
    \label{generatorL}
L f = {\rm Tr} (D^2 \Phi)^{-1} D^2 f - \langle \nabla f, \nabla W(\nabla \Phi) \rangle.
\end{equation}
Similarly, the generator for $(\mathbb{R}^n,\sigma, D^2\Psi)$ takes the form
\begin{equation}
    \label{generatorL*}
L^* g = {\rm Tr} (D^2 \Psi)^{-1} D^2 g - \langle \nabla g, \nabla V(\nabla \Psi) \rangle.
\end{equation}

The following particular case of spaces $(\mathbb{R}^n,\mu, D^2\Phi)$, $(\mathbb{R}^n,\sigma, D^2\Psi)$, where 
$$
\mu =  \frac{e^{-\Phi} dx}{\int_{\mathbb{R}^n} e^{-\Phi} dx}
$$
and 
$$
\mu^* : = \sigma =  \frac{e^{-\Phi(\nabla \Phi^*)} \det D^2 \Phi^* dy}{\int_{\mathbb{R}^n} e^{-\Phi(\nabla \Phi^*)} \det D^2 \Phi^*  dy}.
$$
is of special interest. 
We remark that $\mu$ is the  "moment measure" for $\mu^*$ (see \cite{Cordero-Klartag}, \cite{FM}, \cite{Klartag-lcm}, \cite{KoKo}).

We observe that in this case the Brascamp--Lieb-type inequality
$$
{\rm Var}_{\mu} f \le C \int_{_{\mathbb{R}^n}} \langle (D^2 \Phi)^{-1} \nabla f, \nabla f \rangle d\mu
$$
is precisely the Poincar{\'e} inequality for the differential operator $L$/metric-measure space $(\mathbb{R}^n,\mu, D^2 \Phi)$.
Since $x \to \nabla \Phi(x)$ is a measure preserving isometry, the above inequality is equivalent to the inequality
$$
{\rm Var}_{\mu^*} g \le C \int_{_{\mathbb{R}^n}} \langle (D^2 \Phi^*)^{-1} \nabla g, \nabla g \rangle d\mu^*.
$$
Taking into account that $V=\Phi+C_1$ and $W = \Phi(\nabla \Phi^*) - \log \det D^2 \Phi^*+C_2$, we see that
 (\ref{generatorL*}) is simplified to
$$
L^*g =  (D^2 \Phi^*)^{-1} D^2 g - \langle \nabla g, y\rangle.
$$
Note that (\ref{generatorL}) takes a more complicated  form
\begin{equation}
\label{1002}
Lf = {\rm Tr} (D^2 \Phi)^{-1} D^2 f - \langle \nabla f, (D^2 \Phi)^{-1} \nabla \Phi\rangle  - \sum_{i=1}^n
{\rm Tr} (D^2 \Phi)^{-1}  (D^2 \Phi_{e_i}) \cdot \langle (D^2 \Phi)^{-1} \nabla f, e_i \rangle.
\end{equation}
The eigenfunctions $f=\Phi_{e_i}$ and  $g=y_i$   of $L$ and $L^*$, respectively, correspond to the eigenvalue $-1$. This value is precisely the first non-zero eigenvalue, and the corresponding spectral gap inequality is the Brascamp--Lieb inequality.

We explain below that the duality between $(\mathbb{R}^n,\mu, D^2 \Phi)$ and $(\mathbb{R}^n,\sigma, D^2 \Psi)$ implemented by the  optimal transportation $T = \nabla \Phi$ has a complete spherical analog (at least for symmetric measures and functions). 
To this end, we consider a couple of probability measures
$$
\nu = \frac{e^{-v} d\varphi}{\int_{\mathbb{S}^{n-1}}e^{-v} d\varphi}, \ \ 
\tau = \frac{e^{-w} d\theta}{\int_{\mathbb{S}^{n-1}}e^{-w} d\theta}
$$
on $\mathbb{S}^{n-1}$, symmetric with respect to the origin.
Let $$T \colon \mathbb{S}^{n-1} \to \mathbb{S}^{n-1}$$ be the optimal transportation mapping solving  the Monge--Kantorovich problem with the cost function
\[
c(x,y):=
\begin{cases}
\log \langle x, y \rangle, & \langle x, y \rangle > 0,\\
+\infty, & \langle x, y \rangle \le 0.
\end{cases}
\]
(see  {\cite{Oliker}, \cite{Bertrand}), and pushing forward measure $\nu$ onto $\tau$.
According to \cite{Oliker}, \cite{Bertrand}, this mapping exists and has the form
$$
T(\varphi) = \frac{\phi \cdot \varphi + \nabla_{\mathbb{S}^{n-1}} \phi}{\sqrt{\phi^2 + |\nabla_{\mathbb{S}^{n-1}} \phi|^2}}, \ \ \varphi \in \mathbb{S}^{n-1},$$
for some even function $\phi \colon \mathbb{S}^{n-1} \to (0,+\infty)$, which admits a $1$-homogeneous convex extension to $\mathbb{R}^n$: $\phi(x) = |x| \phi(\varphi)$.
In particular, if $T$ is sufficiently smooth, one can use the following change of variables formula (see \cite{BLYZ}, Lemma 9.5.3):
\begin{equation}
\label{chvar}
\frac{e^{-v}}{\int_{\mathbb{S}^{n-1}}e^{-v}d\varphi} = \frac{e^{-w(T)}}{\int_{\mathbb{S}^{n-1}}e^{-w} d\theta} 
 \frac{\phi\det D^2 \phi}{(\phi^2 + |\nabla_{\mathbb{S}^{n-1}}\phi|^2)^{\frac{n}{2}}}.
 \end{equation}

Consider the metric measure space
$(\mathbb{S}^{n-1},\nu, g_{\phi})$, equipped with the probability measure
$\nu $
and the metric
$$
g_{\phi} =  {\rm I} + \frac{\nabla^2_{\mathbb{S}^{n-1}} \phi}{\phi} = \frac{D^2 \phi}{\phi}.
$$
The spherical analog of the Legendre transform  is given by the Young transform
$$
h(y) = \sup_{x \in \mathbb{R}^n} \frac{\langle x, y \rangle}{\phi(x)}
$$
and the corresponding  Fenchel--Moreau-type
duality reads as
$\phi(x) = \sup_{y\in \mathbb{R}^{n}} \frac{\langle x,y \rangle}{h(y)}$.

We note that $\phi$ is the Minkowski functional and $h$ is the support function of 
 the following convex symmetric set 
$$
K = \big\{x: \phi(x) \le 1\big\}.
$$
 Oliker \cite{Oliker} observed that the cost function $\log \langle x, y \rangle$ is related to the Alexandrov problem in convex geometry. A relation of this problem to the log-Brunn--Minkowski inequality was established by the first-named author \cite{Kolesnikov-sphere}. We note that the existence of a solution to the Alexandrov problem can be proved only under some geometric assumptions  on measures (see \cite{BFR}, \cite{Oliker}, \cite{Bertrand}), but for symmetric measures with densities (this is exactly our case) all these assumptions
are fulfilled. Indeed, for the well-posedness of the optimal transportation problem with probability measures $\mu_1, \mu_2$, where at least  one measure (say, $\mu_2$) has a density with respect to the uniform measure, it is sufficient to show that
\begin{equation}
    \label{Acond}
\mu_1(A) < \mu_2(A_{\pi/2})
\end{equation}
for every spherically convex set $A$. Here $A_{\pi/2} = \{y: d(y,A)<\pi/2\}$ and $d$ is the standard spherical distance. The sufficiency of this condition was proved by Alexandrov in the  case of uniform measure and by Bertrand  in the general case (see \cite{Bertrand}, explanation in the end of Section 4). If $\mu_1, \mu_2$ are even  measures with positive bounded densities, then (\ref{Acond}) is easy to show. Indeed, note that  by the symmetry $\mu_2(A_{\pi/2}) \ge 1/2$ for every nonempty $A$, because any $A_{\pi/2}$ includes a half-sphere (up to a set of  measure zero). Thus (\ref{Acond}) automatically holds if $\mu_1(A)<1/2$. Let $\mu_1(A) \ge 1/2$ and take an arbitrary $y \notin A_{\pi/2}$. Such $y$ exists only if $A$ coincides (up to measure zero) with the half-sphere; hence $\mu_1(A) =1/2$. In any case,  $A_{\pi/2}=\mathbb{S}^{n-1}$ up to a set of measure zero; thus $\mu_1(A_{\pi/2})=1$ and condition (\ref{Acond}) holds.

See also new developments around the more general Gauss image problem in \cite{Semenov}.

By the symmetry of the cost function $\log \langle x,y \rangle$, the inverse mapping $T^{-1}(\theta)$  is given by
$$
T^{-1}(\theta) = \frac{h \cdot \theta + \nabla_{\mathbb{S}^{n-1}} h}{\sqrt{h^2 + |\nabla_{\mathbb{S}^{n-1}} h|^2}}
$$
and this is precisely the $\log \langle x, y \rangle$-optimal mapping sending $\tau$ onto $\nu$.
Moreover, we will show below (see Section 3) that 
$$
g_{h} = {\rm I} + \frac{\nabla^2_{\mathbb{S}^{n-1}} h}{h} = \frac{D^2 h}{h}
$$
is precisely the push-forward of $g_{\phi}$ under $T$.
Thus the metric-measure isomorphism between spaces $(\mathbb{S}^{n-1},\nu,g_{\phi})$ and $(\mathbb{S}^{n-1},\tau,g_{h})$ is realized by the optimal transportation mapping for the cost function 
$\log \langle x, y \rangle.
$
Finally, the weighted Laplacian for $(\mathbb{S}^{n-1},\nu, g_{\phi})$ has been computed in \cite{Kolesnikov-sphere}:
\begin{align}
    \label{Lsphere}
    L g & =  {\rm Tr} \bigl[ \phi (D^2 \phi)^{-1} \nabla^2_{\mathbb{S}^{n-1}} g \bigr]
    + 2 \langle (D^2 \phi)^{-1} \nabla_{\mathbb{S}^{n-1}} \phi, \nabla_{\mathbb{S}^{n-1}} g\rangle 
    \\& \nonumber -
    \langle \nabla_{\mathbb{S}^{n-1}} w(T) + nT, \nabla_{\mathbb{S}^{n-1}} g  \rangle \frac{ \phi}{\sqrt{\phi^2 + |\nabla_{\mathbb{S}^{n-1}} \phi|^2}}.
\end{align}

We collect all these objects and the relationships between them in the table below.

{
\begin{center}
\begin{tabular}{ | m{3cm}|   m{7cm}| m{7cm}|} 
\hline
    & $\mathbb{R}^n$  &  $\mathbb{S}^{n-1}$ \\ 
  \hline
   Metric-measure space & $(\mathbb{R}^n, \mu = \frac{e^{-V} dx}{\int_{\mathbb{R}^n} e^{-V} dx}, D^2 \Phi)$  &  $(\mathbb{S}^{n-1},\nu = \frac{e^{-v}d\varphi}{\int_{\mathbb{S}^{n-1}}e^{-v}d\varphi},  \frac{D^2 \phi}{\phi})$ \\ 
   \hline
     Cost function
& $\langle x, y \rangle $ & $\log\langle x, y \rangle $ \\
  \hline Optimal transportation
     &  $T(x) = \nabla \Phi(x)$ &   $T(\varphi) = \frac{\phi \cdot \varphi + \nabla_{\mathbb{S}^{n-1}}\phi}{\sqrt{\phi^2 + |\nabla_{\mathbb{S}^{n-1}}\phi|^2}}$ \\ 
\hline
    Dirichlet form on the mm space
& $\int_{\mathbb{R}^n} \langle (D^2 \Phi)^{-1} \nabla f, \nabla f \rangle d\mu$ & $\int_{\mathbb{S}^{n-1}} \langle \phi (D^2 \phi)^{-1} \nabla_{\mathbb{S}^{n-1}} f,  \nabla_{\mathbb{S}^{n-1}} f \rangle d\nu$\\
\hline
    Dual (push-forward) measure
&  $\sigma=\frac{e^{-W}dy}{\int_{\mathbb{R}^n}e^{-W}dy}$ & $\tau = \frac{e^{-w} d\theta}{\int_{\mathbb{S}^{n-1}}e^{-w} d\theta}$  \\
\hline
     Change of variables formula
& $\frac{e^{-V}}{\int_{\mathbb{R}^n} e^{-V} dx} = \frac{e^{-W(\nabla \Phi)}}{\int_{\mathbb{R}^n}e^{-W}dy} \det D^2 \Phi$ & $\frac{e^{-v}}{\int_{\mathbb{S}^{n-1}}e^{-v}d\varphi} = \frac{e^{-w(T)}}{\int_{\mathbb{S}^{n-1}}e^{-w} d\theta} 
 \frac{\phi \det D^2 \phi}{(\phi^2 + |\nabla_{\mathbb{S}^{n-1}}\phi|^2)^{\frac{n}{2}}}
 $ \\
\hline
     Dual potential 
& $\Psi(y) = \Phi^*(y) = \sup_{x \in \mathbb{R}^n} \bigl( \langle x, y \rangle - \Phi(x)\bigr)$ & $h(y) = \sup_{x \in \mathbb{R}^{n}}\frac{\langle x,y \rangle}{\phi(x)}$ \\
\hline
     Dual metric-measure space
& $(\mathbb{R}^n, \sigma, D^2 \Psi)$ & $(\mathbb{S}^{n-1},\tau,  \frac{D^2 h}{h})$\\
\hline
    Backward optimal transportation
& $T^{-1}(y) = \nabla \Phi^*(y)$ & $ T^{-1}(\theta) 
 = \frac{h \cdot \theta + \nabla_{\mathbb{S}^{n-1}}h}{\sqrt{h^2 + |\nabla_{\mathbb{S}^{n-1}}h|^2}}$\\
\hline
     Weighted Laplacian of the mm space
&  $L f = {\rm Tr} (D^2 \Phi)^{-1} D^2 f - \langle \nabla f, \nabla W(T) \rangle$ &  see  (\ref{Lsphere}) \\
  \hline
\end{tabular}
\end{center}
}

The following particular case of the couple of spaces $(\mathbb{S}^{n-1},\nu,g_{\phi}), (\mathbb{S}^{n-1},\tau,g_{h})$ is of special interest. This is a spherical version of the couple $(\mathbb{R}^n,\mu, D^2\Phi)$, $(\mathbb{R}^n,\mu^*, D^2\Psi)$, where 
$
\mu =  \frac{e^{-\Phi} dx}{\int_{\mathbb{R}^n} e^{-\Phi} dx}
$ is the moment measure for $\mu^*$. Given a symmetric convex body $K$ and its  Minkowski functional $\phi=\phi_K$ we define:
\begin{equation}
\label{0802}
\nu = \frac{\frac{d\varphi}{\phi^n}}{\int_{\mathbb{S}^{n-1}}\frac{d\varphi}{\phi^n}}
\end{equation}
and consider the "spherical moment map" $T(\varphi) =  \frac{\phi \cdot \varphi + \nabla_{\mathbb{S}^{n-1}}\phi}{\sqrt{\phi^2 + |\nabla_{\mathbb{S}^{n-1}}\phi|^2}}$.

Then, by the change of variables formula on $\mathbb{S}^{n-1}$, the dual measure 
$\nu^* = \nu \circ T^{-1}$
takes the form
\begin{equation}
    \label{0902}
\tau = \nu^* := \frac{h \det D^2h d\theta}{\int_{\mathbb{S}^{n-1}} h \det D^2h d\theta}.
\end{equation}
Here  $h=h_K$ is the support function of $K$.
 Measure (\ref{0902}) is known as the "cone measure" of  $K$, and  $L^*$ for this particular case is the Hilbert operator of $K$.
It has been computed in \cite{KM-LpBMproblem}:
$
L^*\bigl( \frac{u}{h}\bigr) = {\rm Tr} (D^2 h)^{-1} D^2 u - (n-1) \frac{u}{h},
$ equivalently
$$
L^*g = h \cdot {\rm Tr} (D^2 h)^{-1} \nabla^2_{\mathbb{S}^{n-1}} g + 2 \langle h (D^2h)^{-1} \nabla_{\mathbb{S}^{n-1}} h, \nabla_{\mathbb{S}^{n-1}} g \rangle.
$$
The computation of the Hilbert operator $L$ in terms of the Minkowski functional $\psi$ can be found in \cite{Milman-cageometry}, Section 4.4.

The equivalence between inequalities (\ref{240125}) and (\ref{240126}) is an immediate corollary of the duality between  $(\mathbb{S}^{n-1},\nu,g_{\phi}), (\mathbb{S}^{n-1},\nu^*,g_{h})$ for the special choice of measures given by (\ref{0802}), (\ref{0902}).

The $p$-Brunn--Minkowski conjecture is a question about the first eigenvalue of the Hilbert operator restricted to the set of even functions. Note that the first eigenvalue on the set of all functions corresponds to the standard Brunn--Minkowski inequality, and the eigenfunctions are well-known:
\begin{equation}\label{eigHilb}
    L\bigl(\phi \langle \varphi,e_i\rangle + \phi_{e_i}\bigr) = -(n-1)(\phi \langle \varphi,e_i\rangle + \phi_{e_i}\bigr), \ \ \ \ \
L^*\bigl( \frac{y_i}{h}\bigr) =  - (n-1) \frac{y_i}{h}.
\end{equation}

{
\begin{center}
\begin{tabular}{ | m{3cm}|   m{7cm}| m{7cm}|} 
\hline
     & $\mathbb{R}^n$  &  $\mathbb{S}^{n-1}$ \\ 
  \hline
   Metric-measure space 
   & $(\mathbb{R}^n, \mu = \frac{e^{-\Phi} dx}{\int_{\mathbb{R}^n} e^{-\Phi} dx}, D^2 \Phi)$  &  $(\mathbb{S}^{n-1},\nu  = \frac{\frac{d\varphi}{\phi^n}}{\int_{\mathbb{S}^{n-1}}\frac{d\varphi}{\phi^n}},  \frac{D^2 \phi}{\phi})$ \\ 
   \hline
     Dual (push-forward) measure
& $\mu^* = \frac{e^{-\Phi(\nabla \Phi^*)} \det D^2 \Phi^* dy}{\int_{\mathbb{R}^n} e^{-\Phi(\nabla \Phi^*)} \det D^2 \Phi^*  dy}$ &
 $\nu^* = \frac{h\det D^2h d\theta}{\int_{\mathbb{S}^{n-1}} h \det D^2h d\theta}$ (cone measure)\\
 \hline 
 Weighted Laplacian
     & see (\ref{1002}) &   see \cite{Milman-cageometry} \\ 
     \hline
     Weighted Laplacian in  the dual space
     & 
     $L^* g = {\rm Tr} (D^2 \Phi^*)^{-1} D^2 g - \langle \nabla g, y \rangle$
     & $L^*\bigl( \frac{u}{h}\bigr) = {\rm Tr} (D^2 h)^{-1} D^2 u - (n-1) \frac{u}{h}$ (Hilbert operator)  \\
\hline
Eigenfunctions  for the first eigenvalue & $L \Phi_{x_i} = - \Phi_{x_i}, \ L^* y_i = - y_i$ & see (\ref{eigHilb}) \\
\hline
\end{tabular}
\end{center}
}

Thus, we see that the proof of the local $p$-Brunn--Minkowski inequality can sometimes be reduced to the proof of a strong Brascamp--Lieb inequality for a measure on $\mathbb{R}^n$.
We will exploit this observation throughout the paper. 
As an illustration of this approach, let us revisit  Corollary \ref{lpballres}.

\begin{example}
    Consider  $(\mathbb{R}^n, \mu = \frac{e^{-\Phi} dx}{\int_{\mathbb{R}^n} e^{-\Phi} dx}, D^2 \Phi)$, where $\Phi = \sum_{i=1}^n |x_i|^p, p>1$. 
    We observe that $\mu$ and $\mu^*$ are unconditional product measures.
    As we already observed above,  $L^*(y_i)=-y_i$
    Since $D^2 \Phi^*$ is diagonal, one has  
    $$
L^*(y_i y_j) = - 2 y_i y_j + 2 (D^2 \Phi^*)^{-1}_{ij} = -2 y_i y_j,
    $$
    thus $y_i y_j$ is a symmetric eigenfunction of $L^*$.
    Then, using the unconditionality of $\mu$ and applying a decomposition argument from Section 8 of \cite{CKLR} (see explanations in Section 4), one can show that
    ${\rm Var}_{\mu^*} f \le \max\Bigl(\frac{1}{2},1 - \frac{1}{p} \Bigr) \int_{\mathbb{R}^n}
\langle (D^2 \Phi^*)^{-1} \nabla f, \nabla f \rangle d\mu^*$. Changing the variables back, one gets
    $$
{\rm Var}_{\mu} f \le \Bigl( 1 - \frac{1}{p} \Bigr) \int_{\mathbb{R}^n}
\langle (D^2 \Phi)^{-1} \nabla f, \nabla f \rangle d\mu
    $$
    for $p \ge 2$ and 
      $$
{\rm Var}_{\mu} f \le \frac{1}{2}\int_{\mathbb{R}^n}
\langle (D^2 \Phi)^{-1} \nabla f, \nabla f \rangle d\mu
    $$
    for $1 < p \le  2$. 
Note that $\Phi$ is smooth and strictly convex up to a set of measure zero; thus, all the inequalities above hold (this can be justified by approximations). 

    This is a more precise estimate than the one obtained in \cite{CKLR}.
    Together with Theorem \ref{mainth1} it implies that $l^q_n$-balls  satisfy the local $p$-Brunn--Minkowski inequality for some $p = p(q)<0$. 
\end{example}

To prove the following results, 
we work directly with the Hilbert operator and apply the decomposition argument from \cite{CKLR}. 

\begin{theorem}
\label{k2ij}
    Let $K$ be an unconditional smooth uniformly convex body, $n \ge 3$, and let $\phi(x) = \|x\|_K$ be the Minkowski functional of $K$. Suppose
    $$
\frac{\phi_{e_i e_i}}{\phi^2_{e_i}} 
+
\frac{\phi_{e_j e_j}}{\phi^2_{e_j}} 
\ge  \frac{2n}{n-2} 
\frac{\phi_{e_i e_j}}{\phi_{e_i} \phi_{e_j} } 
    $$
    for all $i \ne j$,  provided $\phi_{e_i} \ne 0,  \phi_{e_j}\ne 0$. Then $K$ satisfies the local log-Brunn--Minkowski inequality (\ref{240125}).

    In particular, the result holds, provided $\phi_{e_i e_j} \le 0$ for all $i \ne j$. 
\end{theorem}

\begin{theorem} \label{pinch-uncond} Let $K$ be an  unconditional smooth uniformly convex body satisfying
$$
h(D^2h)^{-1} \ge \lambda 
$$
for some $\lambda \in (0,1]$, where $h$ is the support function of $K$ (equivalently ${\rm II}_{\partial K}(x) \ge \frac{\lambda}{\langle x, \nu \rangle}$, where $\nu$ is the outer normal to $\partial K$ at $x$). Then $K$ satisfies the local $p$-BM inequality with 
$p = {1-\lambda}$.
\end{theorem}

\begin{remark}
    Inequality $
h(D^2h)^{-1} \ge \lambda 
$ cannot hold for $\lambda >1$, otherwise, we have inequality $\nabla^2 h  + \bigl(1- \frac{1}{\lambda} \bigr) h\le 0$. Hence, $\Delta_{\mathbb{S}^{n-1}} h + (n-1)(1-\frac{1}{\lambda}) h \le 0$. Integrating this with respect to the uniform measure, we find that the average of $h$ is negative, but $h >0$.
\end{remark}

The result of Theorem \ref{pinch-uncond} is a version of pinching estimates, obtained in works of E.~Milman \cite{Milman-cageometry} and 
 Ivaki--E.~Milman \cite{MI1}. The important difference with these results is that in  Theorem \ref{pinch-uncond} only the lower bound is assumed.

 We were also able to recover the pinching estimate of Ivaki and Milman (with a slightly different constant) in a simple way. To this end, we use the standard Bochner  formula on $\mathbb{R}^n$  equipped with a probability log-concave measure $\mu = \frac{e^{-\Phi} dx}{\int_{\mathbb{R}^n} e^{-\Phi} dx}$ and the Euclidean metric (a form of the Bochner identity for the centroaffine  connection on $\mathbb{S}^{n-1}$ was also applied in \cite{Milman-cageometry}, \cite{MI1}). Then we apply Theorem \ref{mainth1}.

\begin{theorem} (Ivaki-E.~Milman, \cite{MI1})
Let $\Phi = \frac{1}{2} r^2 \phi^2(\varphi)$, where  $\phi$ is the Minkowski functional of a smooth convex body $K$. Assume that
$$
\alpha {\rm I}
\le D^2 \Phi \le \beta {\rm I},
$$
where $0 \le \alpha \le \beta $.
Then  $K$ satisfies the local $p$-Brunn--Minkowswki inequality with
$$  p = 1 - n \frac{\alpha}{\beta} - \Bigl( \frac{\alpha}{\beta} \Bigr)^2.
$$
In particular, the local logarithmic Brunn-Minkowski  inequality holds if
$$
\frac{\alpha}{\beta} \ge \frac{\sqrt{n^2+4}-n}{2} \sim \frac{1}{n}.
$$
\end{theorem}

Finally, in the last Section we revisit the connections between the weighted Blaschke--Santal\'o inequality and the strong Brascamp--Lieb/local $p$-Brunn--Minkowski inequality. This was the main subject of paper \cite{CKLR}. Generally, the weighted Blaschke--Santal\'o inequality  always implies a  form of the Brascamp--Lieb-type inequality (the latter is just the corresponding infinitesimal version of the weighted Blaschke--Santal\'o), this was proved in \cite{CKLR}. 
In this way, we prove  some sufficient conditions for the local $p$-Brunn--Minkowski inequality, deriving it from the weighted Blaschke--Santal\'o inequality. Finally, we prove a functional version of the main result of E. Milman from \cite{Milman-cainequality} and obtain  the following characterization of Gaussian measures. 

\begin{theorem}\label{BLgauss}
    Let $\Phi$ be a $2$-homogeneous smooth convex even function, and $\mu = \frac{e^{-\Phi}dx}{\int e^{-\Phi} dx}$ satisfy the strong Brascamp--Lieb inequality
    $$
{\rm Var}_{\mu} f \le \frac{1}{2} \int \langle (D^2 \Phi)^{-1} \nabla f, \nabla f \rangle d \mu
    $$
    for all even $f$.
    Then $\mu$ is a Gaussian measure.
\end{theorem}

\subsection*{Acknowledgements}
The  first
named author was supported by RSF project 25-11-00007. The second and third named authors are jointly supported by NSF-BSF grant DMS-2247834. The second named author is also supported in part at the Technion by a
fellowship from the Lady Davis Foundation. The third named author is also supported by ISF grant 2574/24.

\section{Strong Brascamp--Lieb: homogeneous case}

Everywhere below we assume that 
$$
\mu = \frac{e^{-\Phi}dx}{\int e^{-\Phi}dx}
$$
is a symmetric log-concave probability measure  
with $\alpha$-homogeneous potential $\Phi$:
$$
\Phi(r,\varphi)  = \frac{1}{\alpha} r^{\alpha} \phi^{\alpha}(\varphi).
$$

{\bf Regularity assumptions:} in the theorems and  proofs, we will assume that $\alpha>1$, and $\Phi$ is a
three times continuously differentiable
function, satisfying $\det D^2\Phi >0$ everywhere.

\begin{remark}
    In some concrete examples of this paper, the smoothness/ nondegeneracy assumptions hold only up to  sets of Lebesgue measure zero (e.g., $\Phi = \sum_{i=1}^n |x_i|^p$). In these examples, we will provide some extra justifications.
\end{remark}

We are interested in the {\bf strong  Brascamp--Lieb inequality } 
 \begin{equation}
    \label{BLC}
{\rm Var}_{\mu} f \le C \int_{\mathbb{R}^n} \langle \bigl( D^2 \Phi \bigr)^{-1} \nabla f, \nabla f \rangle d\mu
    \end{equation}
    with value $C<1$ on the set of even functions.

    Observe that the best value of $C$ in this inequality satisfies $C \le 1$ (because for $C=1$ this is the standard  Brascamp--Lieb inequality, which holds for all functions) and $C \ge 1 - \frac{1}{\alpha}$, because 
${\rm Var}_{\mu} f = \Bigl(1 - \frac{1}{\alpha} \Bigr) \int_{\mathbb{R}^n} \langle \bigl( D^2 \Phi \bigr)^{-1} \nabla f, \nabla f \rangle d\mu$ for $f = \Phi$.

\begin{theorem}
\label{BL-BLsph}
Assume that     
\begin{equation}
    \label{BLCalpha}
{\rm Var}_{\mu} f \le C_{\alpha} \int_{\mathbb{R}^n} \langle \bigl( D^2 \Phi \bigr)^{-1} \nabla f, \nabla f \rangle d\mu
    \end{equation}
    for some value $ 1 - \frac{1}{\alpha} \le C_{\alpha} \le  1$ and all even $f$.
Then the probability measure 
 $$
       \nu = \frac{\frac{d \varphi}{\phi^n}}{\int_{\mathbb{S}^{n-1}}\frac{d \varphi}{\phi^n}}
  $$ on $\mathbb{S}^{n-1}$ satisfies 
 the following inequality:
  \begin{equation}
  \label{BLsphere-alpha}{\rm Var}_{\nu} g \le \frac{C^2_{\alpha}}{(n-\alpha)C_{\alpha} + \alpha-1} \int_{\mathbb{S}^{n-1}} \langle \phi\bigl( D^2 \phi \bigr)^{-1} \nabla_{\mathbb{S}^{n-1}}  g, \nabla_{\mathbb{S}^{n-1}}  g \rangle d\nu.
 \end{equation}
 Here $g$ is an arbitrary smooth even function on $\mathbb{S}^{n-1}$ and $D^2 \phi = \phi \cdot \delta_{ij} + \nabla^2_{\mathbb{S}^{n-1}}  \phi$.
\end{theorem}

\begin{proof}
    Let us do some preliminary computations.
Consider the standard spherical coordinates
$$
x = r \cdot \varphi,
$$
where $r = |x|$.
Note that $\nu$ is the image (projection) of $\mu$ onto the sphere under the mapping $x \to \varphi$. 
 We disintegrate measure $\mu$ with respect to $(r,\varphi)$:
 $$
\mu = \nu(d\varphi) \gamma^{\varphi}(dr),
 $$
 where 
 $$ \gamma^{\varphi}(dr) = \frac{\phi^n(\varphi) r^{n-1} e^{- \frac{\phi^{\alpha}(\varphi) r^{\alpha}}{\alpha}} }{\int_0^{\infty} r^{n-1} e^{-\frac{r^{\alpha}}{\alpha}}dr} dr$$
are the corresponding conditional measures.

In what follows, we fix a point $x$ and do some computations in the neighborhood of $x$. The $n$-dimensional frame consists of unit vectors 
$$
(\varphi_0, e_1, \cdots, e_{n-1}),
$$
where $\varphi_0=\varphi$ and $e_i \in TM_{\mathbb{S}^{n-1}}$ (the tangent space to $\mathbb{S}^{n-1}$ at $x$). The partial derivatives $\partial_{\varphi_i} f$ do satisfy
$$
\partial_{e_i} f = \frac{\partial_{\varphi_i} f}{r}
$$
and we write for brevity
$$
\nabla f = f_r \cdot \varphi + \sum_{i=1}^{n-1} f_{e_i} e_i = f_r \cdot \varphi + \nabla_{\mathbb{S}^{n-1}} f =
 f_r \cdot \varphi + \frac{\nabla_{\varphi} f}{r}.
$$

We will apply the following representation for the hessian $D^2 \Phi$ (see Lemma 7.15 in \cite{CKLR}) :  $$
D^2 \Phi = r^{\alpha-2} \phi^{\alpha-1} \left[ {\begin{array}{cccc}
    (\alpha-1) \phi &(\alpha-1) \phi_{\varphi_1}  & \cdots & (\alpha-1)\phi_{\varphi_{n-1}}\\
   (\alpha-1) \phi_{\varphi_1} & b_{1,1} & \cdots & b_{1, n-1}\\
    \vdots & \vdots & b_{i,j}  & \vdots\\
(\alpha-1)\phi_{\varphi_{n-1}} & b_{n-1, 1} & \cdots & b_{n-1, n-1}\\
  \end{array} } \right],
    $$
    where $$B = (b_{i,j}) =  D^2 \phi + (\alpha-1)\frac{\nabla_{\varphi} \phi \times \nabla_{\varphi} \phi}{\phi}.$$
    One can easily compute the inverse hessian:
$$
(D^2 \Phi)^{-1} = \frac{1}{r^{\alpha-2} \phi^{\alpha-1}} \left[ {\begin{array}{cccc}
    \frac{1}{(\alpha-1) \phi}
    +  \frac{\langle (D^2 \phi)^{-1} \nabla_{\varphi} \phi, \nabla_{\varphi} \phi \rangle}{\phi^2} & - \bigl( \frac{(D^2 \phi)^{-1} \nabla_{\varphi} \phi}{\phi} \bigr)_{1}  & \cdots & - \bigl( \frac{(D^2 \phi)^{-1} \nabla_{\varphi} \phi}{\phi} \bigr)_{n-1} \\
   - \bigl( \frac{(D^2 \phi)^{-1} \nabla_{\varphi} \phi}{\phi} \bigr)_{1}  & (D^2 \phi)^{-1}_{1,1} & \cdots & (D^2 \phi)^{-1}_{1, n-1}\\
    \vdots & \vdots & (D^2 \phi)^{-1}_{i,j}  & \vdots\\
- \bigl( \frac{(D^2 \phi)^{-1} \nabla_{\varphi} \phi}{\phi} \bigr)_{n-1}  & (D^2 \phi)^{-1}_{n-1, 1} & \cdots & (D^2 \phi)^{-1}_{n-1, n-1}\\
  \end{array} } \right].
$$

In particular, one has the following expression for
$\langle (D^2 \Phi)^{-1} \nabla f, \nabla f \rangle$:

\begin{equation}
\label{hess-formula}
\langle (D^2 \Phi)^{-1} \nabla f, \nabla f \rangle
= \frac{1}{r^{\alpha-2} \phi^{\alpha}} \Bigl[ \frac{f^2_r}{\alpha-1}
+  \phi
\big\langle (D^2 \phi)^{-1} \Bigl(\frac{\nabla_{\varphi} f}{r}  - f_r \frac{\nabla_{\varphi}\phi}{\phi}\Bigr), \Bigl( \frac{\nabla_{\varphi} f}{r}  - f_r \frac{\nabla_{\varphi}\phi}{\phi}\Bigr) \big\rangle\Bigr].
\end{equation}

Let us show that (\ref{BLCalpha}) implies (\ref{BLsphere-alpha}).
Take a function 
$$
f= g(\varphi) {r^{k} \phi^{k}(\varphi)},
$$
where parameter $k$ will be chosen later. We assume that $\int f d\mu=0$.

Note that
$$
f_r = k g r^{k-1} \phi^{k}, \ \nabla_{\varphi} f = 
\nabla_{\varphi} g {r^{k} \phi^{k}} + kg r^{k} \phi^{k-1} \nabla_{\varphi} \phi.
$$
In particular, $\frac{\nabla_{\varphi} f}{r}  - f_r \frac{{\nabla_{\varphi}\phi}}{\phi} =  \nabla_{\varphi} g \cdot \phi^{k} r^{k-1} =  \nabla_{\mathbb{S}^{n-1}} g \cdot \phi^{k} r^{k-1}  $.
Thus
\begin{equation}
    \label{hessff-hess}
\langle (D^2 \Phi)^{-1} \nabla f, \nabla f \rangle
=  \Bigl[ \frac{k^2}{\alpha-1} g^2 
+  \phi
\big\langle (D^2 \phi)^{-1} \nabla_{\mathbb{S}^{n-1}} g , \nabla_{\mathbb{S}^{n-1}} g \big\rangle \Bigr] \phi^{2k-\alpha} r^{2k-\alpha}.
\end{equation}
Finally, we compute
$$
0 = \int f d\mu = \int g \bigl( \int_0^{\infty} r^{k}\phi^{k} d \gamma^{\varphi} \bigr)d\nu = \int g d\nu \cdot \frac{\int_0^{\infty} t^{n+k-1} e^{-\frac{t^{\alpha}}{\alpha}} dt}{\int_0^{\infty} t^{n-1} e^{-\frac{t^{\alpha}}{\alpha}} dt}.
$$
In particular, $\int g d\nu =0$. One has
$$
{\rm Var}_{\mu} f = \int f^2 d\mu = 
\int g^2 \bigl( \int_0^{\infty} r^{2k}\phi^{2k} d \gamma^{\varphi} \bigr)d\nu
= \int g^2 d\nu \cdot \frac{\int_0^{\infty} t^{n+2k-1} e^{-\frac{t^{\alpha}}{\alpha}} dt}{\int_0^{\infty} t^{n-1} e^{-\frac{t^{\alpha}}{\alpha}} dt}.
$$
In the same way, we compute
$$
\int 
\langle (D^2 \Phi)^{-1} \nabla f, \nabla f \rangle d\mu =   \frac{\int_0^{\infty} t^{n+2k-\alpha - 1} e^{-\frac{t^{\alpha}}{\alpha}} dt}{\int_0^{\infty} t^{n-1} e^{-\frac{t^{\alpha}}{\alpha}} dt} \Bigl[ \frac{ k^2}{\alpha-1} \int g^2 d\nu
+  \int \phi \big\langle (D^2 \phi)^{-1} \nabla_{\mathbb{S}^{n-1}} g , \nabla_{\mathbb{S}^{n-1}} g \big\rangle d\nu \Bigr].
$$
Plugging these relations into (\ref{BLCalpha}) one gets 
$$\int g^2 d\nu \cdot \frac{\int_0^{\infty} t^{n+2k-1} e^{-\frac{t^{\alpha}}{\alpha}} dt}{\int_0^{\infty} t^{n+2k-\alpha-1} e^{-\frac{t^{\alpha}}{\alpha}} dt}
\le 
\frac{ C_{\alpha} k^2}{\alpha-1} \int g^2 d\nu
+ C_{\alpha} \int \phi \big\langle (D^2 \phi)^{-1} \nabla_{\mathbb{S}^{n-1}} g , \nabla_{\mathbb{S}^{n-1}} g \big\rangle d\nu.
$$
Integrating by parts, one gets $\frac{\int_0^{\infty} t^{n+2k-1} e^{-\frac{t^{\alpha}}{\alpha}} dt}{\int_0^{\infty} t^{n+2k-\alpha-1} e^{-\frac{t^{\alpha}}{\alpha}} dt} = n+2k-\alpha$.
Hence
$$
\Bigl[ \frac{n+2k-\alpha}{C_{\alpha}} - \frac{k^2}{\alpha-1}\Bigr]  \int g^2 d\nu
\le  \int \phi \big\langle (D^2 \phi)^{-1} \nabla_{\mathbb{S}^{n-1}} g , \nabla_{\mathbb{S}^{n-1}} g \big\rangle d\nu.
$$
Choosing the optimal value: $k=\frac{\alpha-1}{C_{\alpha}}$ one gets the result.
\end{proof}

\begin{remark}
    In the proof, we have never used that our functions are even. Repeating the proof, one can easily conclude that the standard Brascamp--Lieb inequality with $C_{\alpha}=1$ implies the inequality 
     $${\rm Var}_{\nu} g \le \frac{1}{n-1} \int_{\mathbb{S}^{n-1}} \langle \phi\bigl( D^2 \phi \bigr)^{-1} \nabla_{\mathbb{S}^{n-1}}  g, \nabla_{\mathbb{S}^{n-1}}  g \rangle d\nu
 $$
 for arbitrary regular $g$.
\end{remark}

Similarly, inequality on the sphere implies the strong Brascamp--Lieb inequality.

\begin{theorem}
\label{BMimpliesBL}
Assume that $
       \nu = \frac{\frac{d \varphi}{\phi^n}}{\int_{\mathbb{S}^{n-1}}\frac{d \varphi}{\phi^n}}
  $ on $\mathbb{S}^{n-1}$ satisfies 
 inequality
  \begin{equation}
  \label{BLsphere}
{\rm Var}_{\nu} g \le C_{\nu} \int_{\mathbb{S}^{n-1}} \langle \phi\bigl( D^2 \phi \bigr)^{-1} \nabla_{\mathbb{S}^{n-1}}  g, \nabla_{\mathbb{S}^{n-1}}  g \rangle d\nu
 \end{equation}
on the set of even functions on $\mathbb{S}^{n-1}$.
    Then $\mu$ satisfies the following strong Brascamp--Lieb inequality:
    \begin{equation}
    \label{BL}
{\rm Var}_{\mu} f \le \max\Bigl( 1 - \frac{1}{\alpha}, n C_{\nu}\Bigr) \int_{\mathbb{R}^n} \langle \bigl( D^2 \Phi \bigr)^{-1} \nabla f, \nabla f \rangle d\mu.
    \end{equation}
\end{theorem}
\begin{proof}
Take any  even $f$ with $\int f d\mu=0$
and set 
$$g(\varphi) = \int_0^{\infty} f(\varphi,r) \gamma^{\varphi}(dr).
$$
Clearly $\int g d\nu = \int f d\mu=0$
and for any fixed $\varphi$
one has
$\int (f -g(\varphi))\gamma^{\varphi}(dr) =0$.
We use the following $1$-dimensional estimate, which is the $1$-dimensional version of the main result of \cite{CorRot}:
$$
\int_0^{\infty} (f -g(\varphi))^2 d\gamma^{\varphi}
\le \frac{1}{\alpha \phi^{\alpha}}\int_{0}^{\infty} \frac{(f_r)^2}{r^{\alpha-2}} d \gamma^{\varphi}.
$$
One has
\begin{align*}
{\rm Var}_{\mu} f =\int f^2 d\mu &= 
\int (f -g(\varphi))^2 d\mu + \int g^2 d\nu
\\& \le 
\int \Bigl(\int_0^{\infty} (f -g(\varphi))^2 d \gamma^{\varphi} \Bigr)d\nu + C_{\nu}\int \phi \langle (D^2 \phi)^{-1} \nabla_{\mathbb{S}^{n-1}} g, \nabla_{\mathbb{S}^{n-1}} g \rangle d\nu
\\& \le   \frac{1}{\alpha} 
\int \frac{1}{\phi^{\alpha}}\frac{(f_r)^2}{r^{\alpha-2}} d\mu + C_{\nu}\int \phi \langle (D^2 \phi)^{-1} \nabla_{\mathbb{S}^{n-1}} g, \nabla_{\mathbb{S}^{n-1}} g \rangle d\nu.
\end{align*}
Let us compute $\nabla_{\mathbb{S}^{n-1}} g = {\nabla_{\varphi} g}$. One has
$$
g(\varphi) = \int f d\gamma^{\varphi} = \frac{\int_0^{\infty} f(r,\varphi)\phi^n(\varphi) r^{n-1} e^{- \frac{\phi^{\alpha}(\varphi) r^{\alpha}}{\alpha}} dr}{\int_0^{\infty} r^{n-1} e^{-\frac{r^{\alpha}}{\alpha}}dr} 
= 
\frac{\int_0^{\infty} f(\frac{s}{\phi(\varphi)},\varphi) s^{n-1} e^{- \frac{s^{\alpha}}{\alpha}} ds}{\int_0^{\infty} r^{n-1} e^{-\frac{r^{\alpha}}{\alpha}}dr}. 
$$
Differentiating this formula and changing variables back $s = \phi(\varphi) r$, one easily gets
 $$
\nabla_{\mathbb{S}^{n-1}} g(\varphi) = \nabla_{\varphi} g(\varphi) 
= \int_0^{\infty} r \bigl( \frac{\nabla_{\varphi} f}{r}
- f_r \frac{\nabla_{\varphi} \phi}{\phi}\bigr) \gamma^{\varphi}(dr).
    $$
Finally, by  Cauchy inequality
\begin{align*}
    & \phi \langle (D^2 \phi)^{-1} \nabla_{\mathbb{S}^{n-1}} g, \nabla_{\mathbb{S}^{n-1}} g \rangle 
    \\& =
     \phi \Big\langle (D^2 \phi)^{-1} \int_0^{\infty} r \bigl( \frac{\nabla_{\varphi} f}{r}
- f_r \frac{\nabla_{\varphi} \phi}{\phi}\bigr) \gamma^{\varphi}(dr), \int_0^{\infty} r \bigl( \frac{\nabla_{\varphi} f}{r}
- f_r \frac{\nabla_{\varphi} \phi}{\phi}\bigr) \gamma^{\varphi}(dr)\Big\rangle 
    \\&
    \le
     \int r^{\alpha}  \gamma^{\varphi}(dr) \cdot   \int \frac{1}{r^{\alpha-2}} 
\langle \phi (D^2 \phi)^{-1} \bigl(\frac{\nabla_{\varphi} f}{r}  - f_r \frac{\nabla_{\varphi}\phi}{\phi}\bigr), \bigl( \frac{\nabla_{\varphi} f}{r}  - f_r \frac{\nabla_{\varphi}\phi}{\phi}\bigr) \rangle \gamma^{\varphi}(dr)
\\& = \frac{n}{\phi^{\alpha}}   \int \frac{1}{r^{\alpha-2}} 
\langle \phi (D^2 \phi)^{-1} \bigl(\frac{\nabla_{\varphi} f}{r}  - f_r \frac{\nabla_{\varphi}\phi}{\phi}\bigr), \bigl( \frac{\nabla_{\varphi} f}{r}  - f_r \frac{\nabla_{\varphi}\phi}{\phi}\bigr) \rangle \gamma^{\varphi}(dr).
    \end{align*}
    Plugging this estimate into the inequality for ${\rm Var}_{\mu}(f)$ and using (\ref{hess-formula}), one gets
    \begin{align*}
    {\rm Var}_{\mu} f & \le \frac{1}{\alpha} 
\int \frac{1}{\phi^{\alpha}}\frac{(f_r)^2}{r^{\alpha-2}} d\mu +  n C_{\nu}
  \int \frac{1}{r^{\alpha-2}\phi^{\alpha}} 
\langle \phi (D^2 \phi)^{-1} \bigl(\frac{\nabla_{\varphi} f}{r}  - f_r \frac{\nabla_{\varphi}\phi}{\phi}\bigr), \bigl( \frac{\nabla_{\varphi} f}{r}  - f_r \frac{\nabla_{\varphi}\phi}{\phi}\bigr) \rangle  d\mu 
.
\end{align*}
The result follows from formula (\ref{hess-formula}) for $\langle (D^2 \Phi)^{-1} \nabla f, \nabla f \rangle$.
\end{proof}

\begin{corollary}
Inequality
$$
{\rm Var}_{\nu} g \le \frac{1}{n} \Bigl( 1 - \frac{1}{\alpha}\Bigr)\int_{\mathbb{S}^{n-1}} \langle \phi\bigl( D^2 \phi \bigr)^{-1} \nabla_{\mathbb{S}^{n-1}}  g, \nabla_{\mathbb{S}^{n-1}}  g \rangle d\nu
$$
is equivalent to inequality
$$  
{\rm Var}_{\mu} f \le \Bigl( 1 - \frac{1}{\alpha}\Bigr) \int_{\mathbb{R}^n} \langle \bigl( D^2 \Phi \bigr)^{-1} \nabla f, \nabla f \rangle d\mu.
$$
\end{corollary}

{
\begin{remark}
    In the following section, we will see that inequality (\ref{BLsphere}) with 
    $C_{\nu}= \frac{1}{n}$ is equivalent to the  local log-BM inequality for the set $K$ with Minkowski functional $\phi$. We observe that Theorem {\ref{BMimpliesBL}} can't provide any new information about the Brascamp--Lieb-type inequality for $\mu$ if we only know that  $C_{\nu} \ge {\frac{1}{n}}$. Indeed, in that case, we get the trivial bound for ({\ref{BL}}).

We also observe that inequality with $C_{\mu} = \frac{1}{n-p}$, $p \in [0,1)$
(equivalent to the local $p$-BM inequality) is equivalent to "standard Brascamb--Lieb" inequality 
\begin{equation}
    \label{noncBL}
{\rm Var}_{\mu_p} f \le \int \langle (D^2 \Phi)^{-1} \nabla f, \nabla f \rangle d\mu_p,\end{equation} where $\mu_p = C_p e^{-|x|^p \phi^p(\varphi)} dx$, $p \in [0,1)$, and 
$f$ is even and depends only on $\varphi = \frac{x}{|x|}$.
This can be shown with the use of formulas from the proof when one substitutes 
such $f$ into (\ref{noncBL}) and rewrites the inequality as a relation between integrals over $\mathbb{S}^{n-1}$. Note that measure $\mu_p$ is not log-concave for $p<1$, so (\ref{noncBL}) cannot be true for arbitrary $f$.
\end{remark}
}

\section{"Dual" Brascamp--Lieb inequality and local $p$-Brunn--Minkowski inequality}
\label{dualBL}

Let $\Phi$ be as in the previous section and $\Phi^*$ be its Legendre transform. Note that
$$
\Phi^*(y) = \frac{1}{\beta} |y|^{\beta} h^{\beta}\Bigl(\frac{y}{|y|}\Bigr),
$$
where $\frac{1}{\alpha} + \frac{1}{\beta}=1$ and $\phi, h$ are related by
$$
\phi(x) = \sup_y \frac{\langle x, y \rangle}{h(y)}.
$$

Assume that $\mu$ satisfies (\ref{BLCalpha}).
Let $\mu^*$ be the image of $\mu$ under $x \to \nabla \Phi(x)$. By the change of variables formula 
$$
\mu^* = \frac{1}{C} \det D^2 \Phi^* e^{-\Phi(\nabla \Phi^*)} dy
= \frac{1}{C} \det D^2 \Phi^* e^{-(\beta-1)\Phi^*} dy.
$$
Replacing $f$ with $g(\nabla \Phi)$ in  (\ref{BLCalpha}) we see that (\ref{BLCalpha})
is equivalent to the following inequality for $\mu^*$:
\begin{equation}
    \label{BL*}
{\rm Var}_{\mu^*}(g) \le C_{\alpha} \int \langle (D^2 \Phi^*)^{-1} \nabla g, \nabla g \rangle d \mu^*.
\end{equation}

In the previous section, we verified that (\ref{BLCalpha}) is equivalent (up to a constant) to a certain Poincare-type inequality on $\mathbb{S}^{n-1}$. Similarly, inequality (\ref{BL*}) admits an equivalent form on $\mathbb{S}^{n-1}$.

To find this inequality, we proceed as in the previous section.
First, we note that $\mu^*$
can be disintegrated in the polar coordinates as follows:
$$
\mu^* = c' r^{n(\beta-1)-1} h^{n(\beta-1)+1} \det D^2h\cdot  e^{-\frac{1}{\alpha} (hr)^{\beta}} dr d\theta.
$$
(see Corollary 7.16 of \cite{CKLR}).
More precisely, the following representation holds:
\begin{lemma}
    Let $(r,\theta)$ be the polar coordinate system. Then, the image of $\mu^*$ under the mapping $ y \to \theta = \frac{y}{|y|}$ is the following probability measure 
    $$
\nu^* = \frac{h \det D^2 h}{\int h \det D^2 h d\theta} d\theta
    $$
    on $\mathbb{S}^{n-1}$
    and the corresponding conditional measures have the form
    $$
\gamma^{\theta} = \frac{h^{n(\beta-1)}r^{n(\beta-1)-1} e^{-\frac{1}{\alpha} (hr)^{\beta}}}{\int_0^\infty s^{n(\beta-1)-1} e^{-\frac{1}{\alpha} s^{\beta}} ds} dr.
    $$
\end{lemma}
Let us apply the inequality (\ref{BL*}) to the function
$$
g = u(\theta) (hr)^k
$$
such that $\int_{\mathbb{S}^{n-1}} u d\nu^* =0$.
Note that  (\ref{hessff-hess}) implies
\begin{equation}
    \label{hessff-hess*}
\langle (D^2 \Phi^*)^{-1} \nabla g, \nabla g \rangle
=  \Bigl[ \frac{k^2}{\beta-1} u^2 
+  h
\big\langle (D^2 h)^{-1} \nabla_{\mathbb{S}^{n-1}} u , \nabla_{\mathbb{S}^{n-1}} u \big\rangle \Bigr] h^{2k-\beta} r^{2k-\beta}.
\end{equation}
The identities 
$$
\int_0^{\infty} h^{k} r^{k} d \gamma^{\theta} = \frac{\int_0^\infty s^{n(\beta-1)+k-1} e^{-\frac{1}{\alpha} s^{\beta}} ds}{\int_0^\infty s^{n(\beta-1)-1} e^{-\frac{1}{\alpha} s^{\beta}} ds},
$$
$$
\int_0^{\infty} h^{2k} r^{2k} d \gamma^{\theta} = \frac{\int_0^\infty s^{n(\beta-1)+2k-1} e^{-\frac{1}{\alpha} s^{\beta}} ds}{\int_0^\infty s^{n(\beta-1)-1} e^{-\frac{1}{\alpha} s^{\beta}} ds}
$$
can be obtained by direct computations and the change of variable $s = ht$.
One has
$$
\int g d\mu^* =  \frac{\int_0^\infty s^{n(\beta-1)+k-1} e^{-\frac{1}{\alpha} s^{\beta}} ds}{\int_0^\infty s^{n(\beta-1)-1} e^{-\frac{1}{\alpha} s^{\beta}} ds} \cdot\int u d\nu^*=0,
$$
$$
\int g^2 d\mu^* = \frac{\int_0^\infty s^{n(\beta-1)+2k-1} e^{-\frac{1}{\alpha} s^{\beta}} ds}{\int_0^\infty s^{n(\beta-1)-1} e^{-\frac{1}{\alpha} s^{\beta}} ds} \cdot \int u^2 d\nu^*.
$$
Integrating (\ref{hessff-hess*})
one gets
$$
\int 
\langle (D^2 \Phi^*)^{-1} \nabla g, \nabla g \rangle d\mu^* 
=
 \int \Bigl[ \frac{k^2}{\beta-1} u^2 
+  h
\big\langle (D^2 h)^{-1} \nabla_{\mathbb{S}^{n-1}} u , \nabla_{\mathbb{S}^{n-1}} u \big\rangle \Bigr] d\nu^* \cdot   \frac{\int_0^\infty s^{n(\beta-1)+2k-\beta-1} e^{-\frac{1}{\alpha} s^{\beta}} ds}{\int_0^\infty s^{n(\beta-1)-1} e^{-\frac{1}{\alpha} s^{\beta}} ds}.
$$
Thus, we obtain that inequality 
 (\ref{BL*}) implies
$$
\frac{\int_0^\infty s^{n(\beta-1)+2k-1} e^{-\frac{1}{\alpha} s^{\beta}} ds}{\int_0^\infty s^{n(\beta-1)+2k-\beta-1} e^{-\frac{1}{\alpha} s^{\beta}} ds} \cdot \int u^2 d\nu^*
\le C_{\alpha}  \int \Bigl[ \frac{k^2}{\beta-1} u^2 
+  h
\big\langle (D^2 h)^{-1} \nabla_{\mathbb{S}^{n-1}} u , \nabla_{\mathbb{S}^{n-1}} u \big\rangle \Bigr] d\nu^*. 
$$
Equivalently
$$
\bigl[n(\beta-1)+2k-\beta\bigr] \frac{\alpha}{\beta} \cdot \int u^2 d\nu^*
\le C_{\alpha}  \int \Bigl[ \frac{k^2}{\beta-1} u^2 
+  h
\big\langle (D^2 h)^{-1} \nabla_{\mathbb{S}^{n-1}} u , \nabla_{\mathbb{S}^{n-1}} u \big\rangle \Bigr] d\nu^*. 
$$
Taking into account that $\alpha = \frac{\beta}{\beta-1}$, we get
$$
\Bigl[ n + \frac{2k-\beta}{\beta-1} - \frac{k^2 C_{\alpha}}{\beta-1}\Bigr] \int u^2 d\nu^* \le C_{\alpha} \int h
\big\langle (D^2 h)^{-1} \nabla_{\mathbb{S}^{n-1}} u , \nabla_{\mathbb{S}^{n-1}} u \big\rangle  d\nu^*.
$$
The optimal value of $k$ is $k=\frac{1}{C_{\alpha}}$. One finally obtains
\begin{equation}
\label{pBM-BL}
\frac{1}{C_{\alpha}} \Bigl( n + \frac{\frac{1}{C_{\alpha}} - \beta}{\beta-1} \Bigr) {\rm Var}_{\nu^*} u=
 \int u^2 d \nu^* \le  \int 
     h
\big\langle (D^2 h)^{-1} \nabla_{\mathbb{S}^{n-1}} u , \nabla_{\mathbb{S}^{n-1}} u \big\rangle d\nu^*.
\end{equation}
Substituting $\beta = \frac{\alpha}{\alpha-1}$, we obtain the following result.

\begin{theorem}
Assume that $\mu$ satisfies (\ref{BLCalpha}). Then 
$\nu^*$ satisfies
\begin{equation}
\label{pBM-BL+}
 {\rm Var}_{\nu^*} u \le  \frac{C^2_{\alpha}}{(n-\alpha)C_{\alpha} + \alpha - 1} \int 
     h
\big\langle (D^2 h)^{-1} \nabla_{\mathbb{S}^{n-1}} u , \nabla_{\mathbb{S}^{n-1}} u \big\rangle d\nu^*.
\end{equation}
\end{theorem}

In the next theorem, we show that (\ref{pBM-BL+}) is just another form of the inequality (\ref{BLsphere-alpha}). To this end, let us consider the  optimal transportation problem on $\mathbb{S}^{n-1}$ 
     (see \cite{Oliker}, \cite{Kolesnikov-sphere})
$$
\int_{\mathbb{S}^{n-1} \times \mathbb{S}^{n-1}} c(x,y) d\pi \to \max, \ {\rm Pr}_x(\pi) = m_1, \  {\rm Pr}_y(\pi) = m_2
$$
with the cost function

\begin{equation*}
c(x,y) = 
 \begin{cases}
   \log \langle x, y \rangle, &\text{$\langle x, y \rangle > 0$}\\
  -\infty, &\text{$\langle x, y \rangle \le 0$}
 \end{cases}
\end{equation*}.

Let $(\phi, h)$ be solutions to the corresponding primal and dual problems. They are related by the Legendre-type transform 
\begin{equation}
\label{duality}
\phi(x) = \sup_y \frac{\langle x, y \rangle}{h(y)}, \ \ 
h(y) = \sup_x \frac{\langle x, y \rangle}{\phi(x)}.
\end{equation}

In what follows, we work only with values
$x \in \mathbb{S}^{n-1}$, $y \in \mathbb{S}^{n-1}$, so we use spherical coordinates $\varphi,\theta \in \mathbb{S}^{n-1}$ instead.

Without loss of generality, $\phi$ and $h$ can be viewed as  Minkowski and support functionals
of some symmetric convex body $K$.
The corresponding optimal transportation mappings $S, T$ are given by formula (\ref{sandt}).
Finally, recall the following well-known identities (they can be easily derived from (\ref{duality})).
    \begin{equation}
        \label{phih}
h = \frac{1}{ \sqrt{\phi^2(S) + |\nabla_{\mathbb{S}^{n-1}} \phi(S)|^2}}, \ 
\phi(S) = \frac{1}{ \sqrt{h^2 + |\nabla_{\mathbb{S}^{n-1}} h|^2}}.
 \end{equation}

The following theorem is, in fact, an immediate consequence of the known facts that $T$ pushes forward the measure $\nu^*$ to $\nu$ and the metric $\frac{D^2 h}{h}$ to $\frac{D^2 \phi}{\phi}$, but we provide the proof for the reader's convenience.

\begin{theorem}
\label{mainequiv}
Let $(\phi, h)$ be the Minkowski and  support functions of a symmetric convex body $K$ with a smooth, uniformly convex boundary.
     The probability measure 
 $
       \nu = \frac{\frac{d \varphi}{\phi^n}}{\int_{\mathbb{S}^{n-1}}\frac{d \varphi}{\phi^n}}
  $ on $\mathbb{S}^{n-1}$ satisfies 
  inequality
  \begin{equation}
  \label{BLsphereC}{\rm Var}_{\nu} g \le C \int_{\mathbb{S}^{n-1}} \langle \phi\bigl( D^2 \phi \bigr)^{-1} \nabla_{\mathbb{S}^{n-1}}  g, \nabla_{\mathbb{S}^{n-1}}  g \rangle d\nu
 \end{equation}
 for some $C<1$ and arbitrary even $g \colon \mathbb{S}^{n-1} \to \mathbb{R}$ if and only if the
 measure $
\nu^* = \frac{h \det D^2 h}{\int h \det D^2 h d\theta} d\theta
    $ satisfies the inequality 
 \begin{equation}\label{BMsphereC}
{\rm Var}_{\nu^*} u\le  C \int_{\mathbb{S}^{n-1}} 
     h
\big\langle (D^2 h)^{-1} \nabla_{\mathbb{S}^{n-1}} u , \nabla_{\mathbb{S}^{n-1}} u \big\rangle d\nu^*
\end{equation}
 for arbitrary even $u \colon \mathbb{S}^{n-1} \to \mathbb{R}$.
\end{theorem}
\begin{proof}
    We observe that 
    $\nu^*$ is the image of $\nu$ under $S$
    (equivalently, $\nu$ is the image of $\nu^*$ under $T = S^{-1}$). This follows from the change of variables formula (\ref{chvar})
    and relations (\ref{phih}).
    
Differentiating $S$ and $T$, one obtains
$$
DT(\theta) = \frac{{\rm Pr}^{\bot}_{T(\theta)} D^2 h}{\sqrt{h^2 + |\nabla_{\mathbb{S}^{n-1}} h|^2}}, \ DS(T(\theta)) = \frac{{\rm Pr}^{\bot}_{\theta} D^2 \phi}{\sqrt{\phi^2 + |\nabla_{\mathbb{S}^{n-1}} \phi|^2}},
$$
where ${\rm Pr}^{\bot}_{\theta}$ is the orthogonal projection onto the hyperplane $\{x: \langle x, \theta \rangle = 0 \}$.

To extract (\ref{BLsphereC}) from (\ref{BMsphereC}) we take $u=g(T)$. Since $\nu$ is the image of $\nu^*$ under $T$, it remains to prove that
$$
 h
\big\langle (D^2 h)^{-1} \nabla_{\mathbb{S}^{n-1}} u , \nabla_{\mathbb{S}^{n-1}} u \big\rangle 
= \phi(T) \langle \bigl( D^2 \phi(T) \bigr)^{-1} \nabla_{\mathbb{S}^{n-1}}  g(T), \nabla_{\mathbb{S}^{n-1}}  g(T) \rangle.
$$
Indeed, 
$$
\nabla_{\mathbb{S}^{n-1}} u =
\nabla_{\mathbb{S}^{n-1}}( g(T) ) 
=
(DT)^* (\nabla_{\mathbb{S}^{n-1}} g)(T) 
= \frac{D^2 h}{\sqrt{h^2 + |\nabla_{\mathbb{S}^{n-1}} h|^2}} 
(\nabla_{\mathbb{S}^{n-1}} g) \circ T.
$$
Similarly one gets
$
\nabla_{\mathbb{S}^{n-1}} g 
= \frac{D^2 \phi}{\sqrt{\phi^2 + |\nabla_{\mathbb{S}^{n-1}} \phi|^2}} 
(\nabla_{\mathbb{S}^{n-1}} u) \circ S.
$ Equivalently
$$
(\nabla_{\mathbb{S}^{n-1}} g) (T) 
= \frac{D^2 \phi(T)}{\sqrt{\phi^2(T) + |\nabla_{\mathbb{S}^{n-1}} \phi(T)|^2}} 
\nabla_{\mathbb{S}^{n-1}} u.
$$
Finally, we obtain
\begin{align*}
     h
\big\langle (D^2 h)^{-1} \nabla_{\mathbb{S}^{n-1}} u , \nabla_{\mathbb{S}^{n-1}} u \big\rangle 
& = \frac{h}{\sqrt{h^2 + |\nabla_{\mathbb{S}^{n-1}} h|^2}}
\langle \nabla_{\mathbb{S}^{n-1}} g(T), \nabla_{\mathbb{S}^{n-1}} u \rangle
\\& = \frac{h \sqrt{\phi^2(T) + |\nabla_{\mathbb{S}^{n-1}} \phi(T)|^2}}{\sqrt{h^2 + |\nabla_{\mathbb{S}^{n-1}} h|^2}}
\langle \nabla_{\mathbb{S}^{n-1}} g(T), (D^2 \phi(T))^{-1}  \nabla_{\mathbb{S}^{n-1}} g(T) \rangle
\end{align*}
The desired relation follows from
(\ref{phih}).
\end{proof}

 \begin{remark}
\begin{enumerate}    
    \item For $C_{\alpha}= 1 - \frac{1}{\alpha} = \frac{1}{\beta}$ we get
\begin{equation}\label{nbetaBM}
n\beta {\rm Var}_{\nu^*} u\le  \int 
     h
\big\langle (D^2 h)^{-1} \nabla_{\mathbb{S}^{n-1}} u , \nabla_{\mathbb{S}^{n-1}} u \big\rangle d\nu^*
\end{equation}
and this is equivalent to the inequality
\begin{equation}
    \label{BL*beta}
{\rm Var}_{\mu^*}(g) \le \frac{1}{\beta} \int \langle (D^2 \Phi^*)^{-1} \nabla g, \nabla g \rangle d \mu^*.
\end{equation}
\end{enumerate}
\end{remark}

It was shown in  \cite{KM-LpBMproblem}
(see Proposition 5.2) that  inequality (\ref{pBM-BL}) is a form of the
local $p$-Brunn--Minkowski inequality with
$$
p = n - \frac{1}{C_{\alpha}} \Bigl( n + \frac{\frac{1}{C_{\alpha}} - \beta}{\beta-1} \Bigr)
$$
for convex body $K$ with $h=h_K$.

In particular, for $C_{\alpha}=1$, one obtains $p=1$,
and for $C_{\alpha} = \frac{1}{\beta}$, one obtains $p = n(1-\beta) \le 0$.
In this case, we have the following equivalence:
\begin{corollary}
\label{locallogBM}
    Let $K$ be a symmetric convex body and
    $$
 \phi(x) = |x|_K, h(y) = h_K(y)
    $$
    be the corresponding support function and Minkowski functional.

    The body $K$ satisfies the local $p$-Brunn--Minkowski inequality (\ref{nbetaBM}) with $p<0$ if and only if inequalities 
    (\ref{240126}), (\ref{calpha}), (\ref{cnu}) hold with
    $$
 \alpha = 1 - \frac{n}{p}, \  \beta = 1 - \frac{p}{n},  \ C_{\nu} = \frac{1}{n\beta}.
    $$
\end{corollary}

Besides this, we have the following sufficient condition for the local log-Brunn--Minkowski inequality:
\begin{theorem}
      Let $K$ be a symmetric convex body,
    $$
 \phi(x) = |x|_K, h(y) = h_K(y)
    $$
    be the corresponding support function and Minkowski functional. Assume that $\mu = \frac{1}{C} e^{-\frac{1}{\alpha} |x|_K^{\alpha}}$ satisfies the inequality
    $
{\rm Var}_{\mu} f \le C_{\alpha} \int \langle(D^2 \Phi)^{-1} \nabla f, \nabla f \rangle d\mu.$
If
$$
 n - \frac{1}{C_{\alpha}} \Bigl( n + \frac{\frac{1}{C_{\alpha}} - \beta}{\beta-1} \Bigr) \le 0, 
$$
where $\frac{1}{\alpha} + \frac{1}{\beta}=1$, then 
    $K$ satisfies the local log-Minkowski inequality.
\end{theorem}

Finally, we prove a functional version of the main result from \cite{Milman-cainequality}, which is a characterization of ellipsoids in terms of  the spectral values of the Hilbert operator.

\begin{theorem}
    Let $\Phi$ be a strictly convex, twice continuously differentiable, $2$-homogeneous even function, and $\mu = \frac{e^{-\Phi}dx}{\int e^{-\Phi} dx}$ satisfy the strong Brascamp--Lieb inequality
    $$
{\rm Var}_{\mu} f \le \frac{1}{2} \int \langle (D^2 \Phi)^{-1} \nabla f, \nabla f \rangle d \mu
    $$
    on the set of even functions.
    Then $\mu$ is a Gaussian measure.
\end{theorem}
 \begin{proof}
     By Theorems \ref{BL-BLsph} and \ref{mainequiv}, convex body $K = \{ \Phi \le \frac{1}{2}\}$ satisfies the local $-n$-Brunn--Minkowski inequality. By the main result of \cite{Milman-cainequality}
     the set $K$ is an ellipsoid.
 \end{proof}

\section{Estimating eigenvalues}

\subsection{Estimates in Euclidean space}

Let $K$ be a symmetric convex set and 
$\alpha >1, \beta >1$
satisfy $$\frac{1}{\alpha} + \frac{1}{\beta}=1.$$ 
Consider the following
couples of dual potentials
$$
\Phi(x) = \frac{1}{\alpha}|x|_{K}^{\alpha}, \ \ 
\Phi^*(y) = \frac{1}{\beta} h_K^{\beta}(y),
$$
and "dual" measures $\mu,\mu^*$, where 
$$
\mu = \frac{1}{\int_{\mathbb{R}^n} e^{-\Phi} dx} e^{-\Phi} dx,
$$
$$
\mu^{*} = \mu \circ (\nabla \Phi)^{-1}
= \frac{1}{\int_{\mathbb{R}^n} e^{-\Phi} dx} \det D^2 \Phi^* \cdot e^{-(\beta-1)\Phi^*} dy.
$$

We will prove an inequality of the type (\ref{BL*}), which is a spectral gap inequality for the metric measure space
$
(\mu^*, D^2 \Phi^* )
$
with Dirichlet form $
\Gamma(g) = \int \langle (D^2 \Phi^*)^{-1} \nabla g, \nabla g \rangle d \mu^*
$
and  generator
$$
L^* g = {\rm Tr} (D^2 \Phi^*)^{-1} D^2 g - \langle y, \nabla g \rangle. 
$$

We note that for every $\beta >1$
$$
L^* y_i =-y_i
$$
$$
L^*(y_i y_j) = -2 y_i y_j + 2  (D^2 \Phi^*)^{-1}_{ij}.
$$

\begin{proposition}
\label{uncondprop}
Let $\mu^*$ be an unconditional log-concave measure. 
Assume that the functions $(D^2 \Phi^*)^{-1}_{ii}$ are locally integrable with respect to $\mu^*$
for all $1 \le i \le n$.

Assume further that
\begin{enumerate}
    \item There exists $C_1<1$ such that 
    $
{\rm Var}_{\mu^*} f \le C_1\int_{\mathbb{R}^n}
\langle (D^2 \Phi^*)^{-1} \nabla f, \nabla f \rangle d\mu^*
    $
    for every unconditional function $f$
    \item 
    There exist $C_2<1$ and functions $f_{ij} : \{y_i\ge  0, y_j \ge 0\} \to [0,+\infty)$, $i,j \in \{1,\cdots,n\}, i \ne j$, such that 
  $$
{-L^* f_{ij}}{} \ge \frac{1}{C_2} f_{ij},
  $$
  $f_{ij}$ are strictly positive on $D_{ij} = \{y_i>  0, y_j > 0\}$
  and vanishing on $$\partial D_{ij} = \{y_i=0,y_j \ge 0\} \cup \{y_i \ge  0,y_j =0\}.$$
\end{enumerate}
Then 
 $$
{\rm Var}_{\mu^*} f \le \max(C_1,C_2) \int_{\mathbb{R}^n}
\langle (D^2 \Phi^*)^{-1} \nabla f, \nabla f \rangle d\mu^*
    $$
    for all even locally Lipschitz functions.
\end{proposition}
\begin{proof}
  The proof is a slight extension of the proof of Theorem 8.7 in \cite{CKLR}.
Take a compactly supported smooth test function $f$.  We represent  $f$ as follows:
\begin{equation}
\label{ffa}
f = \sum_{a \in \{0,1\}^n} f_a,
\end{equation}
where every function $f_a(y_1, \cdots,y_n)$ is  $y_i$-even if $a_i=0$ and  $y_i$-odd if $a_i=1$. For instance, if all $a_i$ are zero, then $f_a$ is unconditional. Note that 
if $a=(a_1, \cdots,a_n)$ contains an odd amount of $1$, then $f_a=0$, because $f$ is even.

To obtain this representation, we use the operators 
$$
\sigma_i(y) = (y_1, \cdots, -y_i, \cdots, y_n)
$$ and
$$
T_i^+ f =\frac{f(y) + f(\sigma_i(y))}{2}, 
\ 
T_i^- f =\frac{f(y) - f(\sigma_i(y))}{2}.
$$
Note that $f(y) = T_i^+ f + T_i^- f$, where 
$T_i^+ f$ is $y_i$-even, meaning that
$$
T_i^+ f(\sigma_i(y))= T_i^+ f(y)
$$
and 
$T_i^- f$ is $y_i$-odd: 
$$
T_i^- f(\sigma_i(y))= -T_i^- f(y)
.$$ Consequently, applying the operators
$
T_1^{\pm}, T_2^{\pm}, \cdots, T_n^{\pm},
$
we obtain the representation (\ref{ffa}),
where 
$$
f_a = T_1^{b_1} \cdots T_n^{b_n} f.
$$
Here $b_i =1$, if $a_i = 1$ and $b_i=-1$ if $a_i=0$. 

Next, we note that 
$$
{\rm Var}_{\mu^*} f = \sum_{a \in \{0,1\}^n} {\rm Var}_{\mu^*} f_a.
$$
This is because for every $a\ne b$ $f_a f_b$ is $y_j$-odd at least for some $j$ and measure $\mu^*$ is unconditional.

Let us prove that for all couples of indices $i,j$ and  $a \ne b$, one has
$$
\int (D^2 \Phi^*)^{-1}_{ij} (f_a)_{y_i} (f_b)_{y_j} d\mu^* =0.
$$ Suppose one of the functions (say $f_a$) is $y_k$-even, and $f_{b}$ is $y_k$-odd. 

Let $i \ne j$. In this case, we observe that $(D^2 \Phi^*)^{-1}_{ij}$ is $y_i$-odd and $y_j$-odd and even for other variables if $i\ne j$. One of the indices (say $i$) satisfies $i \ne k$. We observe that if $i=j$, then 
$(D^2 \Phi^*)^{-1}_{ij}$ is unconditional.

Thus, the following options are possible:
\begin{itemize}
    \item $i \ne k, j \ne k$. In this case, $ (D^2 \Phi^*)^{-1}_{ij} $ is $y_k$-even, $(f_a)_{y_i} $ is $y_k$-even,  $(f_b)_{y_j}$ is $y_k$-odd.

    \item $i = k, j \ne k$. In this case, $ (D^2 \Phi^*)^{-1}_{ij} $ is $y_k$-odd, $(f_a)_{y_i} $ is $y_k$-odd,  $(f_b)_{y_j}$ is $y_k$-odd.

     \item $i = j = k$. In this case, $ (D^2 \Phi^*)^{-1}_{ij} $ is $y_k$-even, $(f_a)_{y_i} $ is $y_k$-odd,  $(f_b)_{y_j}$ is $y_k$-even.
\end{itemize}
For all cases, the product $ (D^2 \Phi^*)^{-1}_{ij} (f_a)_{y_i} (f_b)_{y_j}$
is $y_k$-odd; hence its average with respect to $\mu^*$ is zero.
This implies
$$
\int \langle (D^2 \Phi^*)^{-1} \nabla f, \nabla f \rangle d\mu^* = 
 \sum_{a \in \{0,1\}^n} 
 \int \langle (D^2 \Phi^*)^{-1} \nabla f_a, \nabla f_a \rangle d\mu^*.
$$
To prove the statement, it is sufficient (and necessary) to show that 
$$
{\rm Var}_{\mu^*}(f_a) \le \int \langle (D^2 \Phi^*)^{-1} \nabla f_a, \nabla f_a \rangle d\mu^*
$$
for all $a$. For $a=0$ (unconditional case) this holds by assumption. If $a\ne 0$, then $f_a$ is $y_i$-odd and $y_j$-odd for some $i \ne j$. Hence, $f_a$ is vanishing on $\partial D_{ij}$. 
Moreover, let us assume for a while that $f_a$ is vanishing on
a neighborhood of $\partial D_{ij}$. One has
\begin{align*}
    \frac{1}{C_2} \int_{D_{ij}} f^2_a d\mu^* \le 
    & -\int_{D_{ij}} f^2_a \frac{L^* f_{ij}}{f_{ij}} d\mu^* = 
    2 \int_{D_{ij}} \frac{f_a}{f_{ij}} \langle (D^2 \Phi^*)^{-1} \nabla f_a, \nabla f_{ij} \rangle d\mu^* \\& - 
     \int_{D_{ij}} \frac{f^2_a}{f^2_{ij}} \langle (D^2 \Phi^*)^{-1} \nabla f_{ij} \nabla f_{ij} \rangle d\mu^*  \le
     \int_{D_{ij}}  \langle (D^2 \Phi^*)^{-1} \nabla f_a, \nabla f_a \rangle d\mu^* .
\end{align*}
Thus, we completed the proof for functions vanishing  on some neighborhood of $\partial D_{ij}$.

For a compactly supported test function $f$ that is odd in $x_{i}$
and $x_{j}$, we know it vanishes on $\partial D_{ij}$, but not necessarily
on a neighborhood of this set. So we take smooth functions $\xi_{k}:[0,\infty)\to[0,1]$
such that $\xi_{k}(t)=1$ on $[\frac{1}{k},\infty)$, $\xi_{k}(t)=0$
on $\left[0,\frac{1}{2k}\right]$, and $\xi_{k}^{\prime}(t)\le Ck$
for some constant $C>0$. Then 
\[
f_{k}(x)=f(x)\xi_{k}(x_{i})\xi_{k}(x_{j})
\]
 is compactly supported and vanishes in a neighborhood of $\partial D_{ij}$,
and so
\[
\int_{D_{ij}}f_{k}^{2} \ d\mu^{\ast}\le C_{2}\int_{D_{ij}}\big\langle \left(D^{2}\Phi^{\ast}\right)^{-1}\nabla f_{k},\nabla f_{k}\big\rangle \ d\mu^{\ast}.
\]
On the left-hand side, we clearly have 
\[
\int_{D_{ij}}f_{k}^{2} \ d\mu^{\ast}\ge\int_{\left\{ \frac{1}{k}\le x_{i},x_{j}\le k\right\} }f_{k}^{2} \ d\mu^{\ast}=\int_{\left\{ \frac{1}{k}\le x_{i},x_{j}\le k\right\} }f^{2} d\mu^{\ast}\xrightarrow{k\to\infty}\int_{D_{ij}}f^{2} \ d\mu^{\ast}
\]
 by monotone convergence. 

On the right hand side, we may assume $\int_{D_{ij}}\left\langle \left(D^{2}\Phi^{\ast}\right)^{-1}\nabla f,\nabla f\right\rangle d\mu^{\ast}<\infty$;
otherwise, there is nothing to prove. We have 
\[`
\nabla f_{k}(x)=\nabla f(x)\xi_{k}(x_{i})\xi_{k}(x_{j})+f(x)\xi_{k}(x_{j})\xi_{k}^{\prime}(x_{i})e_{i}+f(x)\xi_{k}(x_{i})\xi_{k}^{\prime}(x_{j})e_{j}.
\]

By the dominated convergence theorem, we have 
\[
\int_{D_{ij}}\big\langle \left(D^{2}\Phi^{\ast}\right)^{-1}\nabla f,\nabla f\big\rangle \xi_{k}(x_{i})^{2}\xi_{k}(x_{j})^{2} d\mu^{\ast}\xrightarrow{k\to\infty}\int_{D_{ij}}\big\langle \left(D^{2}\Phi^{\ast}\right)^{-1}\nabla f,\nabla f\big\rangle  d\mu^{\ast}.
\]
 For the second term, we obtain
\[
\int_{D_{ij}}\big\langle \left(D^{2}\Phi^{\ast}\right)^{-1}e_{i},e_{i}\big\rangle f^{2}\xi_{k}(x_{j})^{2}\xi_{k}^{\prime}(x_{i})^{2}  d\mu^{\ast}=\int_{D_{ij}}\left(D^{2}\Phi^{\ast}\right)_{ii}^{-1}f^{2}\xi_{k}(x_{j})^{2}\xi_{k}^{\prime}(x_{i})  d\mu^{\ast}.
\]
 Since $f$ is odd in $x_{i}$ and $x_{j}$, it may be written as $f(x)=x_{i}x_{j}g(x)$
for a smooth function $g$, which is also compactly supported and,
hence, bounded. We also note that $\left|\xi_{k}(x_{j})\right|\le1$,
$\left|\xi_{k}^{\prime}(x_{i})\right|\le Ck$, and $\xi_{k}^{\prime}(x_{i})=0$
for $x_{i}\ge\frac{1}{k}$. It follows that 
\begin{align*}
\left|\int_{D_{ij}}\left(D^{2}\Phi^{\ast}\right)_{ii}^{-1}f^{2}\xi_{k}(x_{j})^{2}\xi_{k}^{\prime}(x_{i}) d\mu^{\ast}\right| & \le\int_{\supp(f)\cap\left\{ x_{i}\le\frac{1}{k}\right\} }\left(D^{2}\Phi^{\ast}\right)_{ii}^{-1}x_{i}^{2}x_{j}^{2}g^{2}\cdot\left(Ck\right)^{2} d\mu^{\ast}\\
 & \le\sup  (g^{2})C^{2}\cdot\int_{\supp(f)\cap\left\{ x_{i}\le\frac{1}{k}\right\} }x_{j}^{2}\left(D^{2}\Phi^{\ast}\right)_{ii}^{-1} d\mu^{\ast}\\
 & \xrightarrow{k\to\infty} 0
\end{align*}
 by the dominated convergence theorem. Here we used the fact that
$\left(D^{2}\Phi^{\ast}\right)_{ii}^{-1}\in L_{loc}^{1}\left(\mu^{\ast}\right)$
, and $x_{j}^{2}$ is bounded on the compact set $\supp(f)$. 

The third term also tends to $0$ as $k\to\infty$ by exchanging the
roles of $i$ and $j$, and then the mixed terms also go to $0$ by the
Cauchy inequality: We use here the fact that 
\[
\left\langle \left\langle F,G\right\rangle \right\rangle =\int\big\langle \left(D^{2}\Phi^{\ast}\right)^{-1}F,G\big\rangle d\mu^{\ast}
\]
 is an inner product (on vector valued functions), so if $\left\langle \left\langle G,G\right\rangle \right\rangle \to0$,
then by the Cauchy inequality, $\left\langle \big\langle F,G\right\rangle \big\rangle $
for any $F$ with $\left\langle \left\langle F,F\right\rangle \right\rangle <\infty$. 

It follows that 
\[
\int_{D_{ij}}\big\langle \left(D^{2}\Phi^{\ast}\right)^{-1}\nabla f_{k},\nabla f_{k}\big\rangle d\mu^{\ast}\xrightarrow{k\to\infty}\int_{D_{ij}}\big\langle \left(D^{2}\Phi^{\ast}\right)^{-1}\nabla f,\nabla f\big\rangle  d\mu^{\ast},
\]
 so the proof is complete.

By symmetry, the integrals over $D_{ij}$  can be replaced by integrals over $\mathbb{R}^n$. The proof is complete.

\end{proof}

\begin{example}
\label{productq}
    Consider the product measure $\mu = C \prod_{i=1}^n e^{-|x_i|^p}dx$.
    Thus $K = B_p$, $\alpha=p$. 
    We know (see Theorem A in \cite{CKLR}) that for every unconditional $g$ one has
$$
{\rm Var}_{\mu} g \le \Bigl( 1- \frac{1}{p}\Bigr) \int \langle (D^2 \Phi)^{-1} \nabla g, \nabla g \rangle  d\mu.
$$
To obtain the estimate for all even functions, let us apply Proposition \ref{uncondprop}.
    Changing variables, we consider the metric measure space $(D^2\Phi^*, \mu^*)$.
    One has $\Phi^*_{ij}=0$ if $i \ne j$. Hence
    $L^*(y_i y_j) = - 2y_i y_j$.

Note that Proposition \ref{uncondprop} is applicable, because 
\[
\left(D^{2}\Phi^{\ast}\right)_{ii}^{-1}=\frac{1}{\Phi_{ii}^{\ast}}=\left|x_{i}\right|^{2-q}.
\]
and 
\begin{align*}
\left(D^{2}\Phi^{\ast}\right)_{ii}^{-1} d\mu^{\ast} & =C\left|x_{i}\right|^{2-q}\cdot\det D^{2}\Phi^{\ast}\cdot e^{-c\Phi^{\ast}} d x\\
 & =C\left|x_{i}\right|^{2-q}\cdot\prod_{j=1}^{n}\left|x_{j}\right|^{q-2}e^{-c\Phi^{\ast}} d x\\
 & =C\prod_{j\ne i}\left|x_{j}\right|^{q-2}\cdot e^{-c\Phi^{\ast}} d x.
\end{align*}
 Since $q-2>-1$, we see that $\left(D^{2}\Phi^{\ast}\right)_{ii}^{-1}\in L_{loc}^{1}\left(\mu^{\ast}\right)$,
so the result of Proposition \ref{uncondprop} applies.

Thus, one proves that
$$
{\rm Var}_{\mu^*} g \le C_p \int \langle (D^2 \Phi^*)^{-1} \nabla g, \nabla g \rangle  d\mu^*,
$$
equivalently
$$
{\rm Var}_{\mu} f \le C_p \int \langle (D^2 \Phi)^{-1} \nabla f, \nabla f \rangle  d\mu,
$$
where $C_p = 1 - \frac{1}{p}$
if $p \ge 2$ and 
$C_p = \frac{1}{2}$ if $1 \le p \le 2$.

This is a more precise estimate than the one obtained in \cite{CKLR}, Theorem C.
\end{example}

\begin{corollary}
    $l^q_n$-balls satisfy the local log-Brunn--Minkowski inequality (see Remark \ref{lpbest}) in any dimension.
\end{corollary}

\subsection{Estimates for unconditional bodies}

In what follows, we work with the metric measure space
$$
\Bigl(\nu^*, \frac{D^2h}{h}\Bigr).
$$
Everywhere in this section, it  will be assumed that $h$ is twice continuously differentiable and 
$D^2 h >0$.

The local $p$-Brunn--Minkowski inequality 
\begin{equation}
    \label{0604}
(n-p) {\rm Var}_{\nu^*} g \le \int_{\mathbb{S}^{n-1}} \langle h(D^2h)^{-1} \nabla_{\mathbb{S}^{n-1}} 
g, \nabla_{\mathbb{S}^{n-1}} g \rangle d\nu^*
\end{equation}
is equivalent to a spectral gap estimate for $
\bigl(\nu^*, \frac{D^2h}{h}\bigr).
$
The corresponding generator $L$ has the form
$$
L^* \Bigl( \frac{u}{h} \Bigr)
= {\rm Tr} (D^2 h)^{-1} D^2 u - (n-1) \frac{u}{h}.
$$
Using that $D^2 y_i =0$, we get that 
functions $\frac{y_i}{h(y)}$ are eigenfunctions of $L$ with eigenvalue $-(n-1)$. This eigenvalue corresponds to the standard Brunn--Minkowski inequality $(p=1)$.

 We need an analog of Proposition \ref{uncondprop} for the operator $L^*$, and this is the following Lemma.
{
\begin{lemma}
\label{decomp}
    The unconditional set $K$ with support function $h$ satisfies the $p$-Brunn--Minkowski inequality for some $p \in  [0,1)$
if and only if inequality (\ref{0604}) holds for every couple of indices $1 \le i \ne j \le n$ and every even $g$ satisfying $g_{\{y_i=0\} \cup 
\{y_j=0\}}=0$.
\end{lemma}
}
\begin{proof}
We  represent $g$ in the form
$$
g = \sum_{a \in \{-1,1\}^n} g_a
$$
in the same way it was done in Proposition \ref{uncondprop} for functions on $\mathbb{R}^n$.
Note that the unconditional component
$g_a$ with $a = (1,1,\cdots,1)$ satisfies (\ref{0604}) because unconditional sets do satisfy the log BM-inequality ($p=0$).
To complete the proof of the statement, we have to show that
$\int g_a g_b d\nu^*=0$  and
$$
\int \langle h(D^2h)^{-1} \nabla_{\mathbb{S}^{n-1}} 
g_a, \nabla_{\mathbb{S}^{n-1}} g_b \rangle d\nu^* =0
$$
for $a \ne b$.

The first equality can be proved in the same way as in Proposition \ref{uncondprop}. To prove the second equality we extend homogeneously
$g$ and $h$ onto $\mathbb{R}^n$:
$$
f(r,\theta)= g(\theta), \  \ \Psi = \frac{1}{2} r^2 h^2(\theta).
$$
Then 
 $$
 \langle h(D^2h)^{-1} \nabla_{\mathbb{S}^{n-1}} 
g_a, \nabla_{\mathbb{S}^{n-1}} g_b \rangle =  r^2 h^2 \langle (D^2 \Psi)^{-1} \nabla f_a, \nabla f_b \rangle
 $$ (see the proof of Theorem \ref{BL-BLsph}). Let $m$ be an  unconditional measure on $\mathbb{R}^n$
such that the spherical projection of $ m$ coincides with $\nu^*$. One has
$$
\int_{\mathbb{S}^{n-1}}  \langle h(D^2h)^{-1} \nabla_{\mathbb{S}^{n-1}} 
g_a, \nabla_{\mathbb{S}^{n-1}} g_b \rangle d\nu^* = \int_{\mathbb{R}^n} r^2 h^2 \langle (D^2 \Psi)^{-1} \nabla f_a, \nabla f_b \rangle dm.
$$
Following the arguments of Proposition \ref{uncondprop} (we note that our assumptions on $h$ immediately imply that Proposition \ref{uncondprop} is applicable), we easily conclude that the right-hand side of this formula is zero.
Clearly, every $g_a$ is vanishing on a couple of hyperplanes unless it is unconditional.
The proof is complete.
\end{proof}

Let us consider the space $$\mathbb{S}^{n-1}_{ij} = \mathbb{S}^{n-1} \cap \{y_i>0, y_j >0 \}$$
and the following function: $$z_i = \frac{y_i}{h}.$$ 
 Note that $z_i$ is positive on $\mathbb{S}^{n-1}_{ij}$.
 One has for $y_i >0$:
 $$
L^* (\log z_i) = \frac{L^* z_i}{z_i} - \frac{1}{z^2_i} \langle h (D^2h)^{-1} \nabla_{\mathbb{S}^{n-1}} z_i, \nabla_{\mathbb{S}^{n-1}} z_i \rangle
= -(n-1) - \frac{1}{z^2_i} \langle h (D^2h)^{-1} \nabla_{\mathbb{S}^{n-1}} z_i, \nabla_{\mathbb{S}^{n-1}} z_i \rangle
 $$
 Note that
$$
\nabla_{\mathbb{S}^{n-1}} z_i
= \nabla  z_i = \frac{e_i}{h} - y_i \frac{\nabla h}{h^2}
= z_i \Bigl(  \frac{e_i}{y_i} -  \frac{\nabla h}{h}\Bigr)
$$
Denote
$$
\hat{e}_i = \frac{e_i}{y_i} -  \frac{\nabla h}{h} 
$$ 
and 
\begin{equation} \label{Qij}
Q_{ij} = \langle h (D^2h)^{-1} \hat{e}_i,\hat{e}_j\rangle.
\end{equation}
Thus, we get the following formula
$$
L^* (\log z_i ) = -(n-1) - Q_{ii}.
$$
Take a function $f$ satisfying $f=0$ on $\{y_i \le 0\} \cup \{y_j \le 0\}$. One has
$$
(n-1) \int f^2 d\nu^*  + \frac{1}{2}\int f^2 (Q_{ii}+ Q_{jj}) d\nu^*= -  \frac{1}{2}\int f^2 L^*(\log z_i + \log z_j) d\nu^*.
$$
Integrate by parts:
\begin{align}
\label{1301}
(n-1) \int f^2 d\nu^*  + \frac{1}{2}\int f^2 (Q_{ii}+ Q_{jj}) d\nu^*
  =  \int f \langle h(D^2h)^{-1} \nabla_{\mathbb{S}^{n-1}} f,  \hat{e}_i + \hat{e}_j \rangle  d\nu^*
\end{align}

We get from (\ref{1301}) 

\begin{align*}
(n-1) \int f^2 d\nu^* & + \frac{1}{2}\int f^2 (Q_{ii}+ Q_{jj}) d\nu^*
\le \int \langle h(D^2h)^{-1} \nabla_{\mathbb{S}^{n-1}} f, \nabla_{\mathbb{S}^{n-1}} f  \rangle d\nu^*
\\& + \frac{1}{4} \int \langle h(D^2h)^{-1} \hat{e}_i + \hat{e}_j , \hat{e}_i + \hat{e}_j \rangle f^2 d\nu^*.
\end{align*}
Hence
$$
(n-1) \int f^2 d\nu^*  + \frac{1}{4}\int f^2 (Q_{ii} - 2 Q_{ij} + Q_{jj}) d\nu^*
\le \int \langle h(D^2h)^{-1} \nabla_{\mathbb{S}^{n-1}} f, \nabla_{\mathbb{S}^{n-1}} f  \rangle d\nu^*
$$
Equivalently,
\begin{align}
\label{2-hyp-est}
(n-1) \int f^2 d\nu^* & + \frac{1}{4}\int f^2  
\langle h(D^2h)^{-1} \hat{e}_i - \hat{e}_j, \hat{e}_i - \hat{e}_j \rangle d\nu^*
\\&  \nonumber \le \int \langle h(D^2h)^{-1} \nabla_{\mathbb{S}^{n-1}} f, \nabla_{\mathbb{S}^{n-1}} f  \rangle d\nu^*.
\end{align}
    
Let us assume that the support function $h$ does satisfy inequality $
h(D^2h)^{-1} \ge \lambda 
$
for some $\lambda >0$. Using the fact that $\frac{1}{y_i^2} + \frac{1}{y_j^2} \ge 4$ provided $y_i^2 + y_j^2 \le 1$, we get
$$
\langle h(D^2h)^{-1} \hat{e}_i - \hat{e}_j , \hat{e}_i - \hat{e}_j \rangle 
\ge \lambda \Bigl| \frac{e_i}{y_j} - \frac{e_j}{y_i} \Bigr|^2 = \lambda \Bigl(\frac{1}{y^2_i} + \frac{1}{y^2_j}\Bigr) \ge 4 \lambda.
$$
Thus, we proved the following theorem.

\begin{theorem}  Let $K$ be an unconditional uniformly convex smooth body satisfying
$$
h(D^2h)^{-1} \ge \lambda 
$$
for some $\lambda \in (0,1]$, where $h$ is the support function of $K$. 
Then
$$
(n-1+ \lambda) {\rm Var}_{\nu^*} f \le  \int \langle h(D^2h)^{-1} \nabla_{\mathbb{S}^{n-1}} f, \nabla_{\mathbb{S}^{n-1}} f  \rangle d\nu^*
$$
for any $f$ vanishing on $\{y_i=0\}$, $\{y_j =0\}$.
In particular, $K$ satisfies the local $p$-BM inequality with 
$p = {1-\lambda}$.
\end{theorem}

Applying again  (\ref{1301}) and the  Cauchy inequality, one gets

\begin{align*}
(n-1) \int f^2 d\nu^* & + \frac{1}{2}\int f^2 (Q_{ii}+ Q_{jj}) d\nu^*
\le \frac{n-1}{n} \int \langle h(D^2h)^{-1} \nabla_{\mathbb{S}^{n-1}} f, \nabla_{\mathbb{S}^{n-1}} f  \rangle d\nu^*
\\& + \frac{n}{4(n-1)} \int (Q_{ii} + 2 Q_{ij} + Q_{jj}) f^2 d\nu^*.
\end{align*}
Thus
$$
n \int f^2 d\nu^*
+ \frac{n-2}{4(n-1)} \int f^2 \Bigl( Q_{ii} + Q_{jj}
- \frac{2n}{n-2} Q_{ij} \Bigr)d\nu^*
\le \int \langle h(D^2h)^{-1} \nabla_{\mathbb{S}^{n-1}} f, \nabla_{\mathbb{S}^{n-1}} f  \rangle d\nu^*
$$
and we get the following result.
\begin{theorem} \label{spe} 
Assume that $n \ge 3$ and $K$ is a convex smooth unconditional body, satisfying 
\begin{equation}
    \label{qij}
Q_{ii} + Q_{jj}
- \frac{2n}{n-2} Q_{ij} \ge 0
\end{equation} for all 
 $1 \le i \ne j \le n$.
Then $K$ satisfies the local log-Brunn--Minkowski inequality.
In particular, the statement holds if $Q_{ij} \le 0$.
\end{theorem}

It follows from the Proposition below that assumption
(\ref{qij}) is equivalent to the following assumption for the Minkowski functional $\phi$:
$$
\frac{\phi_{e_i e_i}}{\phi^2_{e_i}} 
+
\frac{\phi_{e_j e_j}}{\phi^2_{e_j}} 
- \frac{2n}{n-2} 
\frac{\phi_{e_i e_j}}{\phi_{e_i} \phi_{e_j} } \ge 0.
$$

\begin{proposition}
    For every $y \in \mathbb{S}^{n-1}$, the following relation holds
    $$
    Q_{ij}(y) = \frac{\phi(x) \partial^2_{e_i e_j} \phi(x)}{\partial_{e_i} \phi(x) \partial_{e_j} \phi(x)},
    $$
    provided $\partial_{e_i} \phi(x) \ne 0, \partial_{e_j} \phi(x) \ne 0$, where $\phi(x) = |x|_K$ is the Minkowski functional of $K$, $x = h(y) \nabla h(y)$.
\end{proposition}
\begin{proof}
    Let $\Phi(x) = \frac{1}{2} \phi^2(x)$, $H(y) = \frac{1}{2} h^2(y)$. Using the formula for $(D^2 H)^{-1}$ (see the proof of Theorem \ref{BL-BLsph}) we conclude
    $$
Q_{ij} =  \langle h (D^2h)^{-1} \hat{e}_i,\hat{e}_j\rangle
= h^2 \langle  (D^2H)^{-1} \hat{e}_i,\hat{e}_j\rangle.
    $$
    Since $H$ is $2$-homogeneous, one has $D^2 H(y) \cdot y = \nabla H(y)$.
    Thus,
\[
\left(D^{2}H\right)^{-1}\left(\frac{\nabla h}{h}\right)=\left(D^{2}H\right)^{-1}\left(\frac{\nabla H}{h^{2}}\right)=\frac{y}{h^{2}}
\]
 and we may compute explicitly 
\begin{align*}
Q_{ij} & =h^{2}\left\langle \left(D^{2}H\right)^{-1}\left(\frac{e_{i}}{y_{i}}\right)-\frac{y}{h^{2}},\frac{e_{j}}{y_{j}}-\frac{\nabla h}{h}\right\rangle \\
 & =h^{2}\left(\frac{1}{y_{i}y_{j}}\left\langle \left(D^{2}H\right)^{-1}e_{i},e_{j}\right\rangle -\left\langle \left(D^{2}H\right)^{-1}\left(\frac{e_{i}}{y_{i}}\right),\frac{\nabla h}{h}\right\rangle -\cancel{\frac{y_{j}}{h^{2}y_{j}}}+\cancel{\frac{\left\langle y,\nabla h\right\rangle }{h^{3}}}\right)\\
 & =h^{2}\left(\frac{1}{y_{i}y_{j}}\left\langle \left(D^{2}H\right)^{-1}e_{i},e_{j}\right\rangle -\left\langle \left(\frac{e_{i}}{y_{i}}\right),\left(D^{2}H\right)^{-1}\left(\frac{\nabla h}{h}\right)\right\rangle \right)\\
 & =h^{2}\left(\frac{1}{y_{i}y_{j}}\left\langle \left(D^{2}H\right)^{-1}e_{i},e_{j}\right\rangle -\left\langle \left(\frac{e_{i}}{y_{i}}\right),\frac{y}{h^{2}}\right\rangle \right)=\frac{h^{2}}{y_{i}y_{j}}\left\langle \left(D^{2}H\right)^{-1}e_{i},e_{j}\right\rangle -1.
\end{align*}
Now set 
$x=\nabla H(y)=h(y) \nabla h(y)$. Indeed, since $\Phi=H^{\ast}$, we also have
the relations $y=\nabla\Phi(x) = \phi(x)\nabla\phi(x)$ , $\left(D^{2}H(y)\right)^{-1}=\left(D^{2}\Phi\right)(x) $,
as well as 
\[
\Phi(x)=\left\langle x,y\right\rangle -H(y)=\left\langle y,\nabla H(y)\right\rangle -H(y)=2H(y)-H(y)=H(y).
\]
Thus
\begin{align*}
Q_{ij}(y) & =\frac{h^{2}(y)}{y_{i}y_{j}}\left\langle \left(D^{2}H(y)\right)^{-1}e_{i},e_{j}\right\rangle -1
=  \frac{h^{2}(y)}{\phi^2(x) \phi_{e_i}(x) \phi_{e_j}(x)}\left\langle D^{2}\Phi(x) e_{i},e_{j}\right\rangle -1
\\& 
= \frac{H^{2}(y)}{\Phi^2(x) \phi_{e_i}(x) \phi_{e_j}(x)}\Bigl(  \phi(x) \phi_{e_ie_j}(x) + \phi_{e_i}(x) \phi_{e_j}(x)\Bigr) -1
= \frac{\phi(x) \partial^2_{e_i e_j} \phi(x)}{\partial_{e_i} \phi(x) \partial_{e_j} \phi(x)}
\end{align*}
as claimed.
\end{proof}

\subsection{Pinching estimates via Bochner formula}

In this subsection, we prove the so-called  pinching estimates. The result of this section is not new, it was obtained  before by Ivaki and Milman  \cite{MI1}. In the proof, they use deep geometric arguments involving the so-called centroaffine connection on $\mathbb{S}^{n-1}$ and a variant of the Bochner formula on $\mathbb{S}^{n-1}$. Our approach is simpler, and is based on the use of the Euclidean Bochner formula.
We note that the  constant from \cite{MI1} is slightly better, because it was improved by additional arguments.

In what follows, we consider  an even convex smooth potential $V$ and the probability measure 
$\mu = \frac{e^{-V}dx}{\int e^{-V }dx }$.

First, we recall the
following variants of the Bochner formula
on $\mathbb{R}^n$ for $\mu$.

\begin{enumerate}
    \item 
(Euclidean Bochner formula) Bochner formula for the Euclidean metric and measure 
$\mu$. 

In this case, the  weighted Laplacian has the form
$$
{L}_{V} f = \Delta f - \langle \nabla V, \nabla f \rangle
$$
and the corresponding Bochner formula looks as follows:
\begin{equation}
    \label{EuclBoch}
\int (L_{V}u)^2d\mu = \int \langle (D^2 V) \nabla u, \nabla u \rangle d\mu
+ \int {\rm{Tr}} (D^2 u)^2 d\mu.
\end{equation}
\item (Hessian Bochner formula) 
The Bochner formula for the Hessian metric $g  = D^2 \Phi$ and measure 
$\mu$. 

Recall that the image of $\mu$ under $\nabla \Phi$ is denoted by
$\nu = \frac{e^{-W }dx}{\int e^{-W }dx }$
and the weighted Laplacian takes the form
$$
L u = {\rm Tr} \bigl[(D^2 \Phi)^{-1} D^2 u\bigr] - \langle \nabla u, \nabla W(\nabla \Phi)\rangle.
$$
One can prove the following Bochner formula
(see explanations in the Proposition below)
\begin{align}
\label{HessBoch0}
    \int (Lu)^2 d\mu & = \frac{1}{2}\int \langle \bigl( (D^2\Phi)^{-1}D^2 V (D^2\Phi)^{-1}+ D^2 W(\nabla \Phi)\bigr) \nabla u, \nabla u\rangle d\mu  \\&  \nonumber + \frac{1}{4} \int {\rm Tr} \bigl[ (D^2 \Phi)^{-1} G(u)\bigr]^2 d\mu
      + \int \Bigl( {\rm Tr} \bigl( (D^2 \Phi)^{-1} {\rm Hess} (u) \bigr)^2d\mu,
\end{align}
where 
$$
G_{ij}(u) = \langle (D^2 \Phi)^{-1} \nabla \Phi_{ij}, \nabla u \rangle, \ \ {\rm Hess}(u) = D^2 u - \frac{1}{2} G(u).
$$
\end{enumerate}
Note that ${\rm Hess} (u)$ is nothing else but the Hessian of $u$ for the Levi--Civita connection of $D^2 \Phi$.

In fact, the integral formula (\ref{HessBoch0}) can be simplified.
\begin{proposition}
Every smooth, compactly supported function $u$ satisfies
    \begin{equation}
\int (Lu)^2 d \mu = \int \Bigl( {\rm Tr} \bigl( (D^2 \Phi)^{-1} D^2 u \bigr)^2+ \langle D^2W (\nabla \Phi) \nabla u, \nabla u \rangle \Bigr) d \mu.
\end{equation}
\end{proposition}
\begin{proof} 
First, let us explain (\ref{HessBoch0}).
We will use the standard notations from the Riemannian geometry.
Working with the space $(\mu, D^2 \Phi)$, we write
$$
f_i = \partial_{x_i}f, \ f_{ij} = \partial^2_{x_i x_j} f
$$
etc. We also write
$$
W^i = W_{x_i} (\nabla \Phi), \ W^{ij} = W_{x_i x_j} (\nabla \Phi), \cdots
$$
Recall  the standard rules of raising the index and summation in repeating indices. For example:
$$
f_i^j = \Phi^{kj} f_{ki} 
:= \sum_{k} \Phi^{kj} f_{ki},  \ \ \ \ W_i = \Phi_{ik} W^k := \sum_k  \Phi_{ik} W_{x_k}(\nabla \Phi) = \partial_{x_i} (W(\nabla \Phi)).
$$
Here, $\Phi_{ij} = (D^2 \Phi)_{ij},\Phi^{ij} = (D^2 \Phi)^{-1}_{ij}$
Formula (\ref{HessBoch0}) is the standard Bochner formula for $(\mu, D^2 \Phi)$.
Indeed, it is well known that
$
{\rm Hess}(u) = u_{ij} - \frac{1}{2} u^k \Phi_{ijk} 
$
and the  Bakry-Emery tensor has the form
$
 \frac{1}{4} \Phi_{iab} \Phi_{j}^{ab} + \frac{1}{2} V_{ij} +  \frac{1}{2} W_{ij}
$ (see computations in \cite{Kolesnikov2014}). These two formulas imply (\ref{HessBoch0}).

Next, we compute
\begin{align*}
\Gamma_2(u) & =  \| u_{ij} - \frac{1}{2} \Phi_{ijk} u^k  \|^2 + \frac{1}{4} \Phi_{iab} \Phi_{j}^{ab} u^i u^j  + \frac{1}{2} V_{ij}u^i u^j +  \frac{1}{2} W_{ij} u^i u^j
\\&
= u_{ij} u^{ij} - u_{ij} u_k \Phi^{ijk} + \frac{1}{2} \Phi_{iab} \Phi_{j}^{ab} u^i u^j  + \frac{1}{2} V_{ij}u^i u^j +  \frac{1}{2} W_{ij} u^i u^j.
\end{align*}

 First, we show that for every smooth and compactly supported function $u$, one has 
\begin{equation}
\label{uijuk}
2\int u_{ij} u_k \Phi^{ijk} d\mu = \int \Bigl(  \Phi_{iab} \Phi_{j}^{ab} u^i u^j  +  V_{ij}u^i u^j -   W_{ij} u^i u^j \Bigr) d\mu.
\end{equation}
Indeed, let us prove (\ref{uijuk}). Integrate by parts 
\begin{align*}
\int u_{ij} u_k \Phi^{ijk} d\mu  & = \int \partial_{x_j} (u_i)  \bigl( u_k \Phi^{ijk} e^{-V}\bigr) dx = - \int u_i \bigl[ u_k \Phi^{ijk} e^{-V}\bigr]_{x_j} dx\\& = - \int u_{i} \bigl( u_{jk}  - V_j  u_k\bigr) \Phi^{ijk} d\mu - \int u_i u_k \partial_{x_j} \bigl( \Phi^{ijk} \bigr) d \mu. 
\end{align*}
Hence
$$
2 \int u_{ij} u_k \Phi^{ijk} d\mu  = \int u_i u_k V_j \Phi^{ijk} d \mu
- \int u_i u_k \partial_{x_j} \bigl( \Phi^{ia} \Phi^{jb} \Phi^{kc} \Phi_{abc} \bigr) d \mu.
$$
Next, we observe that
$$
 u_i u_k \partial_{x_j} \bigl( \Phi^{ia} \Phi^{jb} \Phi^{kc} \Phi_{abc} \bigr) 
= - 2 u_i u_k \Phi^i_{ab} \Phi^{kab}  + u_i u_k   \Phi^{ia} \Phi^{kc}  \bigl( \partial_{x_j} \Phi^{jb} \Phi_{abc} + \Phi^{jb} \Phi_{abcj} \bigr).
$$
Let us use the following identities (see \cite{Kolesnikov2014}):
$$
\partial_{x_j} \Phi^{jb} = V^b - W^b
$$
$$
L \Phi_{ac} = \Phi^{jb} \Phi_{abcj} - W_k \Phi^k_{ack} = -V_{ac} + W_{ac} + \Phi_a^{kl} \Phi_{ckl}.
$$
One has
$$
 \partial_{x_j} \Phi^{jb} \Phi_{abc} + \Phi^{jb} \Phi_{abcj} = V_k \Phi^k_{ac} -V_{ac} + W_{ac} + \Phi_a^{kl} \Phi_{ckl}.
$$
Finally,
$$
u_i u_k \partial_{x_j} \bigl( \Phi^{ia} \Phi^{jb} \Phi^{kc} \Phi_{abc} \bigr)  = \Bigl(  V_j \Phi^{ikj} -V^{ik} + W^{ik} - \Phi^{iab} \Phi^{k}_{ab}\Bigr) u_i u_k 
$$
and we get (\ref{uijuk}).
Plugging (\ref{uijuk}) into the expression for $\Gamma_2$, one gets 
$$
\int \Gamma_2(u) d \mu = \int \Bigl( u_{ij} u^{ij} + W_{ij} u^i u^j \Bigr) d \mu.
$$
This completes the proof.
\end{proof}

\begin{corollary}
In particular, for the space
$\mu^* = \frac{e^{-\Phi(\nabla \Phi^*)} \det D^2 \Phi^* dy}{\int e^{-\Phi}dx}$ with metric 
$D^2 \Phi^*$ this result reads as
\begin{equation}
\label{HessBoch*}
\int (L^*u)^2 d \mu^* = \int \Bigl( {\rm Tr} \bigl( (D^2 \Phi^*)^{-1} D^2 u \bigr)^2+ \langle( D^2 \Phi^*)^{-1}  \nabla u, \nabla u \rangle \Bigr) d \mu^*,
\end{equation}
where $L^*u = {\rm Tr} (D^2 \Phi^*)^{-1} D^2 u -\langle y,\nabla u \rangle$.
\end{corollary}

\begin{remark} (Bochner-type formula on $\mathbb{S}^{n-1}$)
E. Milman \cite{Milman-cageometry} proved the following Bochner-type formula on $\mathbb{S}^{n-1}$ :
\begin{equation}
\label{caBochner}
\int (L^* u)^2d\nu^* = \int {\rm Tr} \Bigl[ h (D^2h)^{-1} {\rm Hess}^*(u) \Bigr]^2d\nu^* + (n-2) \int h \langle (D^2 h)^{-1} \nabla_{\mathbb{S}^{n-1}} u , \nabla_{\mathbb{S}^{n-1}} u\rangle d\nu^*.
\end{equation}
Here  
$
\nu^* = \frac{h \det D^2h d\theta}{\int_{\mathbb{S}^{n-1}} h \det D^2h d\theta}
$ is the cone measure for $K$, $h=h_K$ is the support function of $K$,
$$
 L^* u  =  {\rm Tr} \bigl[ h (D^2 h)^{-1} \nabla^2_{\mathbb{S}^{n-1}} u \bigr]
    + 2 \langle (D^2 h)^{-1} \nabla_{\mathbb{S}^{n-1}} h, \nabla_{\mathbb{S}^{n-1}} u\rangle 
    =  {\rm Tr} \bigl[ h(D^2 h)^{-1} {\rm Hess}^*(u) \bigr]
$$
is the Hilbert operator and
$$
{\rm Hess}^*(u)_{ij} = (\nabla^2_{\mathbb{S}^{n-1}} u)_{ij} + (\log h)_i u_j + (\log h)_j u_i. 
$$
Formula (\ref{caBochner}) admits a nice interpretation in terms of affine geometry. It turns out that for a special (not Levi--Civita) connection $\tilde{\nabla}$ on  
$\mathbb{S}^{n-1}$, the operator $L^*$ and the tensor $\rm {Hess}^*$ can be viewed as a Laplacian and a Hessian of the corresponding affine structure, and (\ref{caBochner}) is just a non-Riemannian variant of the Bochner formula in its standard form.

We observe that there is  a certain similarity between (\ref{HessBoch*}) and (\ref{caBochner}), but we don't know whether these formulas are related.
\end{remark}

\begin{theorem}
Let $\Phi = \frac{1}{2} r^2 \varphi^2(\theta)$, where  $\varphi$ is the Minkowski functional of a symmetric convex body $K$. Assume that $\Phi$  satisfies
$$
\alpha {\rm I}
\le D^2 \Phi \le \beta {\rm I},
$$
where $0 \le \alpha \le \beta $.
Then  $K$ satisfies the local $p$-Brunn-Minkowski inequality with
$$  p = 1 - n \Bigl(\frac{\alpha}{\beta} \Bigr) - \Bigl( \frac{\alpha}{\beta} \Bigr)^2.
$$
In particular, local logarithmic Brunn-Minkowski inequality holds if
$$
\frac{\alpha}{\beta} \ge \frac{\sqrt{n^2+4}-n}{2} \sim \frac{1}{n}.
$$
\end{theorem}

\begin{remark}
    This recovers the result of Ivaki and Milman \cite{MI1} on pinching estimates (but with a slightly worse constant).
\end{remark}

\begin{proof}
Consider the probability measure 
$\mu = \frac{e^{-\Phi}dx}{\int e^{-\Phi } dx}$. In what follows 
we can apply either formula (\ref{HessBoch0})
or formula (\ref{EuclBoch}), both give the same result. For the sake of simplicity let us use the standard (Euclidean) Bochner formula (\ref{EuclBoch}) with $V=\Phi$.

One has for every smooth complactly supported even $u$:
$$
 \int \| D^2 u\|^2 d\mu
 \ge \alpha \int \sum_{i=1}^n \langle (D^2 \Phi)^{-1} \nabla u_{x_i}, \nabla u_{x_i} \rangle d\mu \ge \alpha \int  \sum_{i=1}^n u^2_{x_i} d\mu \ge \frac{\alpha}{\beta}
 \int \langle (D^2 \Phi) \nabla u, \nabla u \rangle d\mu.
$$
Here we use assumptions on $\Phi$ and the Brascamp--Lieb inequality 
$$
\int u^2_{x_i}d\mu = {\rm Var}_{\mu} u_{x_i} \le 
\sum_{i=1}^n \langle (D^2 \Phi)^{-1} \nabla u_{x_i}, \nabla u_{x_i} \rangle d\mu 
$$
(we use  relation $\int u_{x_i} d\mu =0$, this is because $u_{x_i}$ is odd).
Then (\ref{EuclBoch}) implies
$$
\int (L_{\Phi}u)^2d\mu \ge \frac{\alpha+\beta}{\beta}
\int \langle (D^2 \Phi) \nabla u, \nabla u \rangle d\mu.
$$
This implies the strong Brascamp--Lieb
inequality ( see, for instance, the arguments in \cite{KM})
$$
{\rm Var}_{\mu} f \le \frac{\beta}{\alpha+\beta}
\int \langle (D^2 \Phi)^{-1} \nabla f, \nabla f \rangle d\mu.
$$
The result follows from this estimate and Theorem \ref{BL-BLsph}.
\end{proof}

\section{The Blaschke--Santal\'o inequality revisited}

Let $\Phi$ be a convex, even, and $p$-homogeneous ($p>1$) function on $\mathbb{R}^n$. Assume in addition that $\Phi$ is at least twice continuously differentiable.
We say that $\Phi$
satisfies the generalized  Blaschke--Santal\'o inequality, if every  even $f$ satisfies 
\begin{equation}
\label{BS}
\int e^{-f(x)} dx \left(\int e^{-\frac{1}{p-1}f^*(\nabla \Phi(x))} dx\right)^{p-1}
\le \left(\int e^{-\Phi(x)}dx\right)^p.
\end{equation}

Relations between the Brascamp--Lieb inequality and the  generalized Blaschke--Santal\'o inequality were studied in \cite{CKLR}.
The infinitesimal version of (\ref{BS}) (see  \cite{CKLR}) is precisely the strong Brascamp--Lieb inequality
\begin{equation}
\label{pBM1004}
{\rm Var}_{\mu}f \le \Bigl(1 - \frac{1}{p} \Bigr) \int \langle (D^2 \Phi)^{-1} \nabla f, \nabla f \rangle d\mu,
\end{equation}
where
$\mu = \frac{e^{-\Phi}dx}{\int e^{-\Phi}dx}$.

The following sufficient condition for (\ref{BS}) has been obtained in \cite{CKLR} using the symmetrization method. Note that (\ref{BS}) cannot always be true, because it implies (\ref{pBM1004}), which is equivalent to a $q$-Brunn--Minkowski inequality for $\{\Phi \le c\}$ with some negative $q$, which is well-known to be false in general.
See other counterexamples in \cite{CKLR}.

\begin{theorem}\label{CKLR-th} {(Colesanti--Kolesnikov--Livshyts--Rotem, \cite{CKLR}, Theorem A)}
Let $p > 1$ and let $\Phi$ be an even strictly convex $p$-homogeneous $C^2$-function on $\R^n$. Assume that $\Phi$ is an unconditional function, and that the function
$$
x=(x_1,\dots,x_n)\mapsto \Phi\bigl(x^{\frac{1}{p}}_1,...,x^{\frac{1}{p}}_n\bigr)
$$ 
is concave in 
$
\R^n_+=\{(x_1,\dots,x_n)\colon x_i\ge 0, \ \ 1 \le i \le n\}.
$
Assume, in addition, that  for every coordinate hyperplane $H$, with unit normal $e$, and for every $x'\in H$, the function $r \colon[0,+\infty)\to\R$ defined by
$$
r(t)=\det D^2 \Phi^*(x'+te)
$$
is decreasing.
 Then inequality (\ref{BS})
holds for every even convex $f$.
\end{theorem}

\begin{remark}
    In fact, previous theorem holds for any proper even function with values in $(-\infty,+\infty]$. Indeed, replacing such $f$ with its convex envelope increases the value of the functional.
\end{remark}

\begin{corollary}
    Under the assumptions of the previous theorem, the level sets of $\Phi$  satisfy the $q$-Brunn--Minkowski inequality with $ q = - \frac{n}{p-1}$.
\end{corollary}

\begin{example}
    Let  $$
\Phi(x) = c |x|_q^{p} = c \Bigl(\sum_{i=1}^n
    |x_i|^q\bigr)^{\frac{p}{q}}.$$
    and $p \ge q >1$.
    Then every even function $f$ does satisfy the inequality (\ref{BS}).
\end{example}

\begin{remark}
    Moreover, it was shown in \cite{CKLR} that
    (\ref{BS}) holds for unconditional functions under the assumption $p \ge q >1$.
\end{remark}

In addition to that, we showed that (\ref{BS}) is equivalent to a certain inequality for sets.

\begin{proposition}
\label{set-func}
    Let $p > 1$ and let $\Phi$ be an even strictly convex $p$-homogeneous $C^2$ function on $\R^n$. 
    Inequality (\ref{BS}) holds for an arbitrary convex proper function $\Phi$ if and only if the inequality 
    \begin{equation}
    \label{KKo}
|K|\cdot|\nabla \Phi^* (K^o)|^{p-1}\leq \left|\left\{\Phi\leq \frac{1}{p}\right\}\right|^p
\end{equation}
holds for an arbitrary compact convex body $K$. 

If inequality (\ref{KKo}) holds, equality is attained when $K$ is a level set of $\Phi$ : $K = \{ \Phi \le \alpha \}$. 
\end{proposition}

As a corollary of these two results, we obtain an unusual isoperimetric-type inequality in which the minimizers are not round. In fact, this is a novel isoperimetric property of $l^p_n$-balls for $p\geq 2.$

\begin{corollary}
Suppose $p\geq 2.$ Let $K$ be a symmetric convex body in $\R^n$, $n\geq 2$. Then
    $$|K|\left(\int_{K^{\circ}} \prod_{i=1}^n |x_i|^{\frac{2-p}{p-1}} dx\right)^{p-1}\leq (p-1)^{n(p-1)}|B^n_p|^p,$$
    with equality when $K=B^n_p.$
\end{corollary}

\begin{remark}
E. Milman studied in  \cite{Milman-cainequality} the following analog of the 
Blaschke--Santal\'o functional for sets:
$$
F_{\mu,p}(K) = \frac{1}{p} \frac{\int h^p_Kd\mu}{|K|^{\frac{p}{n}}},
$$
where $K$ is a symmetric set and $\mu$ is a measure on $\mathbb{S}^{n-1}$. In particular,
he proved certain non-uniqueness results for the $p$-BM problem with negative $p$, using properties of this functional (see also relevant observations about non-uniqueness  based on the use of the functional Blaschke--Santal\'o inequality in \cite{CKLR}). The functional $F_{\mu,p}$ is relevant to our BS-functional, but we don't study it here. 
\end{remark}

One can use Theorem \ref{CKLR-th} to derive various inequalities of Brascamp--Lieb type. However, we provide below alternative (and more direct) sufficient conditions for the strong Brascamp--Lieb inequality, which are applicable in a more general situation.

\begin{theorem}
Let $\Phi$ be an unconditional, strictly convex, $p$-homogeneous $C^3$-function.
 Assume that the function
$$
x=(x_1,\dots,x_n)\mapsto \Phi\bigl(x^{\frac{1}{p}}_1,...,x^{\frac{1}{p}}_n\bigr)
$$ 
is concave in 
$
\R^n_+=\{(x_1,\dots,x_n)\colon x_i\ge 0, \ \ 1 \le i \le n\}.
$
Finally, assume that for some $a \in (-\infty, \frac{1}{p-1})$
 every coordinate hyperplane $H$, with unit normal $e$, and for every $y'\in H$, the function $r \colon[0, +\infty) \to\R$ defined by
$$
r(t)=e^{-a \Phi^*(y'+et)}\det D^2 \Phi^*(y'+te)
$$
is decreasing.
Then $\mu = \frac{e^{-\Phi}dx}{\int e^{-\Phi}dx}$ satisfies the inequality 
$$
{\rm Var}_{\mu} f \le 
\max \Bigl( 1- \frac{1}{p}, \frac{p-1}{p-a(p-1)} \Bigr) \int \langle (D^2 \Phi)^{-1} \nabla f, \nabla f \rangle d\mu.
$$
\end{theorem}
\begin{proof}
Note that the concavity assumption implies the strong Brascamp--Lieb inequality for $\mu$ and unconditional functions with the constant $1-\frac{1}{p}=\frac{1}{q}$
(see Theorem A in \cite{CKLR}).

Since $r(t)$ is decreasing, one has   $r'(t)\le 0$. Thus, the assumption of the theorem is equivalent to the following:
$$
\partial_{e_i} (\log \det D^2 \Phi^* - a \Phi^*) = {\rm Tr} (D^2 \Phi^*)^{-1} D^2 \Phi^*_{e_i} - a \Phi^*_{e_i} \le 0
$$
for all $y$ satisfying $y_i \ge 0$. Next, we observe that
\begin{align*}
L^*(\Phi^*_{e_i}) & = {\rm Tr} (D^2 \Phi^*)^{-1} D^2 \Phi^*_{e_i}  - \langle y, \nabla \Phi^*_{e_i}(y) \rangle
= {\rm Tr} (D^2 \Phi^*)^{-1} D^2 \Phi^*_{e_i}  - \langle D^2 \Phi^*(y) y, e_i \rangle
\\& =  {\rm Tr} (D^2 \Phi^*)^{-1} D^2 \Phi^*_{e_i}  - (q-1) \Phi^*_{e_i}\le (a-q+1) \Phi^*_{e_i}.
\end{align*}
Moreover, taking into account that $L^*(y_j) = -y_j$, we get that for $y $, satisfying  $y_i >0, y_j >0$, one has
$$
L^*(\Phi^*_{e_i}(y) y_j) = \Phi^*_{e_i}(y) L^*(y_j)  + L^*(\Phi^*_{e_i}(y)) y_j + 2 \langle (D^2 \Phi^*)^{-1} \nabla \Phi^*_{e_i}, e_j \rangle
\le (a-q) \Phi^*_{e_i} y_j.
$$
Note that the function $\Phi^*_{e_i}(y) y_j$ is vanishing on $\{y_i=0\} \cup \{y_j=0\}$. Thus, by Proposition \ref{uncondprop}, we get that $\mu^*$ satisfies the inequality 
$$
{\rm Var}_{\mu^*} g \le \max\Bigl( \frac{1}{q}, \frac{1}{q-a} \Bigr) \int  \bigl \langle (D^2 \Phi^*)^{-1} \nabla g, \nabla g \bigr \rangle d\mu^*
$$  This inequality is equivalent to the claimed inequality.
\end{proof}

\begin{remark}
    The assumption of concavity of $x \to \Phi(x^{1/p}_1,\cdots,x^{1/p}_n)$ in the previous theorem is needed only to prove the strong Brascamp--Lieb inequality on the set of unconditional functions. If one intends to prove the local log-BM inequality for level sets $\{\Phi \le c\}$, this assumption can be relaxed, since  the local   log-BM inequality  is known for 
    the unconditional functions and sets.  
\end{remark}

\end{document}